\documentclass{amsart}
\usepackage[margin=1.5in]{geometry}
\def\draftdate{\today}

\usepackage{amsmath}
\usepackage{amsfonts}
\usepackage{amssymb}
\usepackage{amsthm}
\usepackage{comment}
\usepackage{epsfig}
\usepackage{psfrag}
\usepackage{mathrsfs}
\usepackage{amscd}
\usepackage[all]{xy}
\usepackage{rotating}
\usepackage{lscape}
\usepackage{amsbsy}
\usepackage{verbatim}
\usepackage{moreverb}
\usepackage[colorlinks=true, linkcolor=blue, citecolor=blue]{hyperref}
\usepackage{color}
\usepackage[titletoc]{appendix} 
\usepackage{enumitem}
\usepackage{scalefnt}
\usepackage{graphicx}
\usepackage{comment}
\usepackage{aliascnt}

\makeatletter
\newtheorem*{rep@theorem}{\rep@title}
\newcommand{\newreptheorem}[2]{%
\newenvironment{rep#1}[1]{%
 \def\rep@title{#2 \ref{##1}}%
 \begin{rep@theorem}}%
 {\end{rep@theorem}}}
\makeatother

\newtheorem{thm}{Theorem}[section]
\newreptheorem{thm}{Theorem}
\newreptheorem{prop}{Proposition}
\newreptheorem{lem}{Lemma}
\newreptheorem{defn}{Definition}

\newaliascnt{lem}{thm}
\newaliascnt{prop}{thm}
\newaliascnt{cor}{thm}
\newtheorem{lem}[lem]{Lemma}
\newtheorem{prop}[prop]{Proposition}
\newtheorem{cor}[cor]{Corollary}
\newreptheorem{cor}{Corollary}
\aliascntresetthe{lem}
\aliascntresetthe{prop}
\aliascntresetthe{cor}

\theoremstyle{definition}
\newtheorem{defn}[thm]{Definition}
\theoremstyle{definition}
\newtheorem{ex}[thm]{Example}
\theoremstyle{definition}
\newtheorem{rem}[thm]{Remark}
\theoremstyle{definition}
\newtheorem{warn}[thm]{Warning}
\theoremstyle{definition}
\newtheorem{nota}[thm]{Notation}

\DeclareMathOperator{\Map}{Map}
\DeclareMathOperator{\Spec}{Spec}
\DeclareMathOperator{\id}{id}
\DeclareMathOperator{\Id}{Id}
\DeclareMathOperator{\hocolim}{hocolim}
\DeclareMathOperator{\holim}{holim}
\DeclareMathOperator{\Hom}{Hom}
\DeclareMathOperator{\Mod}{Mod}

\DeclareMathOperator{\Free}{Free}
\DeclareMathOperator{\Res}{Res}

\DeclareMathOperator{\Cobar}{Cobar}
\DeclareMathOperator{\Barc}{Bar}

\DeclareMathOperator{\Oblv}{Oblv}

\DeclareMathOperator{\Tot}{Tot}

\DeclareMathOperator{\alg}{-alg}
\DeclareMathOperator{\Sp}{Sp}
\DeclareMathOperator{\sSet}{sSet}
\DeclareMathOperator{\Ab}{Ab}
\DeclareMathOperator{\Ho}{Ho}
\DeclareMathOperator{\hofib}{hofib}
\DeclareMathOperator{\Top}{Top}
\DeclareMathOperator{\fat}{fat}
\DeclareMathOperator{\inj}{inj}
\DeclareMathOperator{\Cube}{Cube}
\DeclareMathOperator{\TotHofib}{TotHofib}
\DeclareMathOperator{\Sym}{Sym}
\DeclareMathOperator{\gr}{gr}
\DeclareMathOperator{\cmtr}{cmtr}
\DeclareMathOperator{\Ch}{Ch}
\DeclareMathOperator{\coalg}{-coalg}

\newcommand{\nccompl}{\scriptsize \normalfont \text{NC-compl}}
\newcommand{\cmtrcompl}{\scriptsize \normalfont \text{cmtr-compl}}
\newcommand{\abcompl}{\scriptsize \normalfont \text{Ab-compl}}

\newcommand{\Lie}{\scriptsize \normalfont \text{Lie}}
\newcommand{\oo}{\infty}

\mathchardef\dash="2D

\def\cC{\mathcal C}\def\cD{\mathcal D}
\def\cE{\mathcal E}\def\cF{\mathcal F}

\def\cO{\mathcal O}\def\cP{\mathcal P}

\def\LL{\mathbb L}

\def\SS{\mathbb S}

\def\ZZ{\mathbb Z}

\def\sC{\mathscr C}

\def\sX{\mathscr X}
\def\sY{\mathscr Y}

\begin{document}
\scalefont{1.1}

\title{Derived Commutator Complete Algebras and Relative Koszul Duality for Operads}
\author{Lee Cohn}
\address{Department of Mathematics, The University of Texas, Austin, TX \ 78712}
\email{lcohn@math.utexas.edu}
\date{\draftdate}

\begin{abstract}
We prove that a connected commutator (or NC) complete associative algebra can be recovered in the derived setting from its abelianization together with its natural induced structure. Specifically, we prove an equivalence between connected derived commutator complete associative algebras and connected commutative algebras equipped with a coaction of the comonad arising from the adjunction between associative and commutative algebras.  This provides a Koszul dual description of connected derived commutator (or NC) complete associative algebras and furthermore may be interpreted as a theory of relative Koszul duality for the associative operad relative to the commutative operad. We also prove analogous results in the setting of $\cE_n$-algebras. That is, we develop a theory of commutator complete $\cE_n$-algebras and a theory of relative Koszul duality for the $\cE_n$ operad relative to the commutative operad.  We relate the derived commutator filtration on associative and $\cE_n$-algebras to the filtration on the associative and $\cE_n$ operad whose associated graded is the Poisson and shifted Poisson operad $\cP_n$. We argue that the derived commutator filtration on associative algebras (and $\cE_n$-algebras) is a relative analogue of the Goodwillie tower of the identity functor on the model category of associative algebras (and $\cE_n$-algebras) relative to the model category of commutative algebras. 
\end{abstract}

\maketitle
\tableofcontents

\section{Introduction}
\subsection{Motivation}
Koszul duality first arose in algebraic topology via Quillen's rational homotopy theory \cite{Q}. Quillen shows that there is an equivalence of homotopy theories between $1$-connected rational topological spaces, $1$-reduced differential graded Lie algebras, and $2$-reduced differential graded cocommutative coalgebras. Thus, the data of a $1$-connected rational space can be encoded in two Koszul dual structures: a derived Lie algebra and a derived cocommutative coalgebra. The relationship between commutative (co)algebras and Lie algebras was further clarified in Ginzburg and Kapranov's theory of Koszul duality for operads \cite{GK}, where it is shown the commutative and Lie operad are Koszul dual as operads. That is, the Koszul duality between Lie algebras and commutative algebras can be clearly expressed ``one level up" in the language of operads. Specifically, we have that a shifted version of the Lie cooperad arises as $1 \overset{\LL}{\circ}_{Com} 1 \simeq \Barc(1,Com,1)$, the bar construction for the commutative operad.

 In \cite{LX}, Lurie formulates a version of rational homotopy theory in the language of derived algebraic geometry.  He proves there is an equivalence between the $\oo$-category of differential graded Lie algebras in characteristic $0$, and the $\oo$-category of formal moduli problems.  Since a formal moduli problem may be interpreted as the data of the formal neighborhood of a point inside a derived stack, the differential graded Lie algebra associated to a formal moduli problem in \cite{LX} provides a Koszul dual description of this formal neighborhood.   

In fact, \cite{LX} formulates a theory of formal moduli problems over $\cE_1$-algebras (algebras with a associative multiplication up to coherent homotopy) and uses the self-Koszul duality of the $\cE_1$ operad to prove an equivalence between the $\oo$-category of formal moduli problems over $\cE_1$-algebras and the $\oo$-category of $\cE_1$-algebras.  This may be interpreted as two Koszul dual ways to encode the data of the formal neighborhood of a point inside a non-commutative derived $\cE_1$-stack.  

In \cite{K}, Kapranov develops an approach to non-commutative geometry ``in the formal neighborhood" of ordinary commutative geometry.  For example, if $X$ is an affine non-commutative scheme i.e. $X =\Spec(A)$ where $A$ is a discrete associative algebra, then the natural surjection $A \to \Ab(A):=A/(A[A,A]A)$ may be interpreted as an embedding $X_{\Ab} \to X$, where 
$X_{\Ab} := \Spec(A/(A[A,A]A))$ is a commutative affine scheme.  Furthermore, one can complete the associative algebra $A$ with respect to a topology in which iterated commutators like $[[[a_1,a_2],a_3],a_4]$ and $[a_1,a_2][a_3,a_4]$ are small.  Elements of the completion can be thought of as formal commutator series.  The completion with respect to this topology (also called completion with respect to the NC-(co)filtration or commutator (co)filtration) is the formal neighborhood of the commutative space $X_{\Ab}$ inside the non-commutative space $X$. A discrete associative algebra $A = \cO_X$ is called NC-complete (or commutator complete) if the completion of $A$ with respect to the NC-(co)filtration is isomorphic to $A$. Thus, an NC-complete associative algebra $A$ produces a non-commutative thickening of $X_{\Ab} = \Spec(\Ab(A))$. Non-commutative thickenings of a commutative space endow the commutative space with extra structure.  For example, the (non-trivial) existence of a smooth non-commutative thickening of an ordinary manifold $M$ induces extra differential-geometric structure on $M$, e.g. an NC-Atiyah class in $H^1(M, \Omega^2_M \otimes T_M)$. 

It is natural to ask if there is a derived version of such an approach to non-commutative geometry, and if there is way to express or encode the formal neighborhood of a commutative space inside a non-commutative space using Koszul duality.  

The purpose of this paper is to answer both of these questions in the affirmative. That is, we reformulate the theory of NC-filtrations (or commutator filtrations) in a derived context, and develop a theory of relative Koszul duality to encode the formal neighborhood of a derived commutative algebra inside of a derived associative algebra. This relative Koszul dual structure is a comonad $K$ acting on the model category of commutative algebras, and $K$ may be interpreted as $Com \overset{\LL}{\circ}_{Ass} Com$. This can be compared to the situation above where a shifted version of the Lie cooperad arises as $1 \overset{\LL}{\circ}_{Com} 1$. 

There is an operad $\cE_{\oo}$ whose algebras have a commutative and associative multiplication up to coherent homotopy. The natural map of operads $\cE_1 \to \cE_{\oo}$ can be refined to the sequence of operads $$\cE_1 \to \cE_2 \to \cE_3 \to \ldots \cE_n \to \cE_{n+1} \to \ldots \to \cE_{\oo}$$ where $\cE_n$ is the operad of little $n$-cubes.  Hence, the operads $\cE_n$ for $n \ge 2$ interpolate between the $\cE_1$ and $\cE_{\oo}$ operads and have algebras that are ``more commutative" than $\cE_1$-algebras but ``less commutative" than $\cE_{\oo}$-algebras. In particular, many questions about $\cE_1$-algebras (i.e. many questions about non-commutative geometry) have natural $\cE_n$-analogues.

Thus, we develop an analogue of Kapranov's approach to non-commutative geometry in the context of $\cE_n$-algebras for $n \ge 2$.  That is, we formulate a theory of commutator filtrations in the setting of $\cE_n$-algebras and a theory of Koszul duality to encode the formal neighborhood of a derived commutative algebra inside of an $\cE_n$-algebra.

\subsection{An Impressionistic Overview of Results}
We now give an overview of the main results in this paper.  

\begin{warn} 
During this overview, we will ignore many homotopical difficulties that arise in formulating the main results precisely.  For example, we will often not introduce new notation to distinguish left or right derived functors from their original, underived counterparts.  We will also ignore issues of homotopy coherence during this introduction. Problems of homotopy coherence can often be resolved by working in an $\oo$-categorical context.  However, a lack of foundations of $\oo$-operads in the $\oo$-category of chain complexes or the $\oo$-category spectra prevent us from working in an $\oo$-categorical context in this paper.  We will place warnings to highlight several (but not all) of the simplifications being made in service of the exposition of this introduction.
\end{warn}

Let $k$ be a field of characteristic 0, and let $\Ch_k$ be the symmetric monoidal model category of chain complexes.  We will work in the setting of operadic algebras in the model category $\Ch_k$. The model categories of associative algebras and commutative algebras in $\Ch_k$ are denoted by $Ass \alg(\Ch_k)$ and $Com\alg(\Ch_k)$. That is, differential graded associative algebras and differential graded commutative algebras.

\begin{warn}
In the body of this paper, we really work in setting of the model category of $Hk$-linear spectra, $\Mod_{Hk}$, whose details are reviewed in Section \ref{Model Categories}.
\end{warn}

The natural restriction functor 
$$\Res: Com\alg(\Ch_k) \to Ass \alg(\Ch_k)$$
has a left adjoint 
$$\Ab: Ass \alg(\Ch_k) \to Com \alg(\Ch_k).$$
which sends an associative algebra to its abelianization.

This produces a Quillen adjunction
$$\Ab: Ass \alg(\Ch_k) \leftrightarrows Com \alg(\Ch_k) : \Res.$$

Moreover, one can define the functor
$$K = \Ab \Res : Com \alg(\Ch_k) \to Com \alg(\Ch_k)$$

The functor $K$ is a homotopical comonad whose properties are described in Section \ref{Comonad K}.  In particular, there is a well-defined notion of a homotopy $K$-coalgebra in $Com \alg(\Ch_k)$ and there is a category of homotopy $K$-coalgebras in $Com \alg(\Ch_k)$ denoted by $K \coalg(Com \alg(\Ch_k))$. 

\begin{warn}
Studying the homotopical comonad $K$ requires significant attention to problems of homotopy coherence.  We will resolve these coherence issues using the work of Blumberg-Riehl \cite{BR}. 
\end{warn}

There is also well-behaved cobar construction on $K$-coalgebras $Y$ defined as follows.
\begin{defn} Let $Y \in K \coalg(Com \alg(\Ch_k))$.  The cobar construction on $Y$, $\Cobar^{\bullet}_K(Y)$, is the cosimplicial object in $Ass \alg(\Ch_k)$ given by
\[\xymatrix{
\Res Y \ar@<-.5ex>[r] \ar@<.5ex>[r] &  \Res K Y \ar@<1.0ex>[r] \ar@<0.0ex>[r] \ar@<-1.0ex>[r] &  \Res K^2 Y  \cdots
}\]
The totalization of this cosimplicial object will be denoted by $\Tot [\Cobar^{\bullet}_K(Y)]$.
\end{defn}

\begin{rem}
Explicitly computing connectivity estimates of $\Tot [\Cobar^{\bullet}_K(Y)]$ occupies the main technical portion of this work. We use the techniques developed in
\cite{HH}, \cite{CH}, and \cite{CH2} frequently.
\end{rem}

By the formal properties of an adjunction, the abelianization-restriction Quillen adjunction defined above gives rise to an induced adjunction as follows
$$\Ab: Ass \alg(\Ch_k) \leftrightarrows K \coalg (Com \alg(\Ch_k)) : \Tot [\Cobar^{\bullet}_K(-)]$$

We wish to upgrade the adjunction above to an equivalence between associative algebras and commutative algebras equipped with the structure of a $K$-coalgebra.  
However, the model category of associative algebras is too large for the derived abelianization functor $\Ab$ to be conservative, i.e. reflect weak-equivalences of associative algebras.  Thus, we use a derived version of Kapranov's  NC-complete (or commutator complete) associative algebras to extract a  full subcategory of $Ass \alg(\Ch_k)$ for which the derived abelianization functor becomes conservative.  

We first recall the definition of a discrete NC-complete associative algebra following \cite{K}.
\begin{defn} Let $A$ be a discrete associative algebra. The NC-filtration (or commutator filtration) on $A$ is a decreasing filtration $\{F^dA\}_{d \ge 0}$ where
$$ F^dA = \underset{m}{\sum} \underset{i_1+\cdots+i_m-m=d}{\sum} A \cdot A^{\Lie}_{i_1} \cdot A \cdots A \cdot A^{\Lie}_{i_m} \cdot A$$
and 
$$A^{\Lie}_{n} = [A,[A, \ldots [A,A]]\ldots] \quad \mbox{(n times)}.$$
\end{defn}
That is, $F^dA$ consists of expressions containing at least $d$ instances of the commutator bracket in $A$.

Since the terms $\{F^dA\}_{d \ge 0}$ form a decreasing filtration on $A$, the quotients $\{A/F^dA\}_{d \ge 0}$ form a cofiltration or tower on A. 

\begin{defn} The NC-cofiltration of a discrete associative algebra $A$ is the tower
$$A \to \cdots \to A/F^dA \to \cdots \to A/F^1A \to A/F^0A.$$
\end{defn}

\begin{rem}
In general, any decreasing filtration forms a cofiltration or tower, and we will use this terminology throughout this work.
\end{rem}

\begin{defn} A discrete associative algebra $A$ is called NC-complete (or commutator complete) if 
$$A \simeq \underset{\leftarrow}{\lim}A/F^dA.$$
\end{defn}
As explained in \cite{K}, this means that the associative algebra $A$ is a formal thickening of its abelianization $\Ab(A)$.

We now define a derived version of Kapranov's NC-filtration (or commutator filtration). The idea is to first reformulate the definition of the NC-filtration in terms of the associative operad. Observe that the NC-filtration can be defined for all associative algebras and is compatible with morphisms of associative algebras. Since the associative operad is a representing object for the category of associative algebras the NC-filtration must correspond to a specific filtration on the associative operad itself. We emphasize that this filtration on the associative operad is the filtration whose associated graded operad is the Poisson operad. Once the NC-filtration is reformulated in terms of the associative operad and the circle product of symmetric sequences (denoted $\circ$), we can use the derived circle product (denoted $\overset{\LL}{\circ}$) to form a derived NC-filtration.  

We now formulate this idea as follows. Let $Ass$ denote the associative operad, and let $Ass^{\le n}$ be the symmetric sequence consisting of operations in the associative operad with less than or equal to $n$ commutator operations. 

\begin{repdefn}{derived NC-cofiltration}
The derived NC-cofiltration (or derived commutator cofiltration) on $A \in Ass \alg(\Ch_k)$ is the tower
$$A \to \cdots \to Ass^{\le n} \overset{\LL}{\circ}_{Ass}(A) \rightarrow \cdots \rightarrow Ass^{\le 1} \overset{\LL}{\circ}_{Ass}(A) \rightarrow Com \overset{\LL}{\circ}_{Ass}(A)$$
\end{repdefn}
 
\begin{repdefn}{nccomplete}
An associative algebra $A \in Ass \alg(\Ch_k)$ is derived NC-complete (or derived commutator complete) if there is a weak-equivalence
$$A \simeq \underset{n}{\holim} [Ass^{\le n} \overset{\LL}{\circ}_{Ass}(A)].$$
\end{repdefn}

The full subcategory of derived NC-complete algebras is denoted by $$Ass \alg^{\nccompl}(\Ch_k).$$ The following proposition implies there is a relationship between the NC-filtration  and the homotopical comonad $K$.

\begin{repprop}{conservative}
The functor 
$$\Ab(-): Ass \alg^{\nccompl}(\Ch_k) \to K \coalg(Com \alg(\Ch_k))$$
is conservative. 
\end{repprop}

\subsubsection{The Main Theorem for Derived Associative Algebras}

The main theorem of this paper proves that this conservative functor $\Ab(-)$ induces an equivalence between connected NC-complete associative algebras and connected homotopy $K$-coalgebras in $Com \alg(\Ch_k)$. Specifically, we show that the derived unit map and the derived counit map arising from the adjunction
$$\Ab: Ass \alg(\Ch_k) \leftrightarrows K \coalg (Com \alg(\Ch_k)) :\Tot [\Cobar^{\bullet}_K(-)]$$
induces weak equivalences on connected algebras. 

Let $\Ch_k^{>0}$ denote the full subcategory of $\Ch_k$ consisting of 0-connected chain complexes.

\begin{repthm}{main}
 There is an equivalence of homotopical categories
$$Ass \alg^{\nccompl}(\Ch_k^{>0}) \simeq K \coalg(Com \alg(\Ch_k^{>0})).$$ 
Specifically, we show that
\begin{enumerate}
\item If $X \in Ass \alg^{\nccompl}(\Ch_k^{>0})$, then the natural map of $Ass$-algebras 
$$X \to \Tot[\Cobar^{\bullet}_K\Ab(X)]$$
is a weak equivalence.
\item If $Y \in K \coalg(Com \alg(\Ch_k^{>0}))$, then the natural map of homotopy $K$-coalgebras
$$\Ab( \Tot[\Cobar^{\bullet}_KY]) \to Y$$
is a weak equivalence.
\end{enumerate}
\end{repthm}

Thus, the theorem above further strengthens the relationship between the homotopical comonad $K$ and the NC-cofiltration.

\subsubsection{The Main Theorem for $\cE_n$-Algebras}
We also formulate analogous results in the setting of $\cE_n$-algebras.  

Given the natural map of operads $\cE_n \to Com$, the restriction functor 
$$\Res: Com\alg(\Ch_k) \to \cE_n \alg(\Ch_k)$$
has a left adjoint 
$$\Ab: \cE_n \alg(\Ch_k) \to Com \alg(\Ch_k).$$
This again produces a Quillen adjunction and one can define the homotopical comonad
$$K_n = \Ab \Res : Com \alg(\Ch_k) \to Com \alg(\Ch_k)$$
and the category of homotopy $K_n$-coalgebras in $Com \alg(\Ch_k)$ is denoted by 
$$K_n \coalg(Com \alg(\Ch_k)).$$
The comonad $K_n \simeq Com \overset{\LL}{\circ}_{\cE_n} Com$ may be interpreted as the Koszul dual of the $\cE_n$ operad relative to the commutative operad.

Again, we wish to formulate a comparison between $\cE_n$-algebras and commutative algebras equipped with the structure of a homotopy $K_n$-coalgebra.  The model category of $\cE_n$ algebras is too large for the derived abelianization functor $\Ab$ to be conservative, i.e. reflect weak-equivalences of algebras.  Thus, we use an $\cE_n$-version of the commutator filtration to extract a full subcategory of $\cE_n \alg(\Ch_k)$ for which the derived abelianization functor becomes conservative. 
 
Let $\cE_n^{\le k}$ be $k^{th}$ t-structure truncation of each arity of the $\cE_n$ operad. We emphasize that this cofiltration corresponds to
 the filtration of the $\cE_n$ operad whose associated graded operad is the shifted Poisson operad $\cP_n$.  Recall that $\cP_n$ is the homology of the $\cE_n$ operad.

\begin{repdefn}{derived En-cofiltration}
The derived commutator cofiltration on $A \in \cE_n \alg(\Ch_k)$ is the tower
$$A \to \cdots \to {\cE_n}^{\le k} \overset{\LL}{\circ}_{\cE_n}(A) \rightarrow \cdots \rightarrow \cE_n^{\le 1} \overset{\LL}{\circ}_{\cE_n}(A) \rightarrow Com \overset{\LL}{\circ}_{\cE_n}(A).$$
\end{repdefn}
  
\begin{repdefn}{commutator complete}
An algebra $A \in \cE_n \alg(\Ch_k)$ is derived commutator complete if
$$A \simeq \underset{k}{\holim} [\cE_n^{\le k} \overset{\LL}{\circ}_{\cE_n}(A)].$$
\end{repdefn}

The full subcategory of derived commutator complete algebras is denoted $$\cE_n \alg^{\cmtrcompl}(\Ch_k).$$ The following proposition indicates a relationship between the commutator filtration on the $\cE_n$ operad and the homotopical comonad $K_n$.

\begin{repprop}{conservativeEn}
The functor 
$$\Ab: \cE_n \alg^{\cmtrcompl}(\Ch_k) \to K_n \coalg(Com \alg(\Ch_k))$$
is conservative. 
\end{repprop}

The second main theorem of this paper proves that this conservative functor induces an equivalence between connected commutator complete $\cE_n$-algebras and connected homotopy $K_n$-coalgebras.

\begin{repthm}{mainEn}
There is an equivalence of homotopical categories
$$\cE_n \alg^{\cmtrcompl}(\Ch_k^{>0}) \simeq K_n \coalg(Com \alg(\Ch_k^{> 0})).$$ 
Specifically, we show that
\begin{enumerate}
\item If $X \in \cE_n \alg^{\nccompl}(\Ch_k^{>0})$, then the natural map of $\cE_n$-algebras 
$$X \to \Tot[\Cobar^{\bullet}_{K_n}\Ab(X)]$$
is a weak equivalence.
\item If $Y \in K_n \coalg(Com \alg(\Ch_k^{>0}))$, then the natural map of homotopy $K_n$-coalgebras
$$\Ab( \Tot[\Cobar^{\bullet}_{K_n}Y]) \to Y$$
is a weak equivalence.
\end{enumerate}
\end{repthm}

Thus, the theorem above further strengthens the connection between the homotopical comonad $K_n$ and the commutator cofiltration on
$\cE_n$-algebras.

\subsection{Analogy with Goodwillie Calculus}
Let $\cO$ be an operad with $\cO(1) \simeq k$ and $\cO(0) \simeq 0$, and let $\Barc(\cO) \simeq 1 \overset{\LL}{\circ}_{\cO} 1 =: \cO^{\vee}$ be the Koszul dual cooperad of $\cO$.
The arity cofiltration of the operad $\cO$ is the tower 
$$\cO \rightarrow \cdots \to \cO^{\le n} \to \cO^{\le 2} \to \cO^{\le 1}$$ 
where $\cO^{\le n}$ denotes the symmetric sequence 
\[
    \cO^{\le n}(k)= 
\begin{cases}
    \cO(k),& \text{if } k \leq n\\
    0,              & \text{otherwise.}
\end{cases}
\]

\begin{warn} We use the notation $\{ \cO^{\le n} \}_{n \ge 1}$ to denote the arity cofiltration of the operad $\cO$ while using the same notation $\{Ass^{\le n} \}_{n \ge 0}$ to denote the commutator cofiltration of the operad $Ass$ (and similarly for the $\cE_m$ operad).  We hope this does not cause too much confusion. 
\end{warn}

We argue that the relationship between between the commutator cofiltration on the associative (or $\cE_n$) operad and the comonad $K$ (or $K_n$) is a relative analogue of the relationship between the arity cofiltration of the operad $\cO$ and the Koszul dual cooperad $\cO^{\vee} = 1 \overset{\LL}{\circ}_{\cO} 1$. 

We now review this latter relationship. The significance of the arity cofiltration on the operad $\cO$ is that it arises via the Goodwillie tower of the identity functor on the model category of $\cO$-algebras.

Given a (finitary) functor $F:\cC \to \cD$ between two (left proper and cellular or proper and combinatorial) model categories, Goodwillie calculus \cite{G3} produces a tower functors 
$$\cF \to \cdots  \to P_n(\cF) \to \cdots \to P_2(\cF) \to P_1(\cF)$$
where each $P_n(\cF):\cC \to \cD$ is $n$-excisive and is universal with respect to this property. For example, a functor is 1-excisive if it sends homotopy pushouts to homotopy pullbacks.  

The example $\Id: \cO \alg(\Ch_k) \to \cO \alg(\Ch_k)$ produces a tower of functors from $\cO \alg(\Ch_k)$ to itself
$$\Id \to \cdots  \to P_n(\Id) \to \cdots \to P_2(\Id) \to P_1(\Id).$$
Furthermore, this tower of functors produces a cofiltration on each object of the model category $\cO \alg(\Ch_k)$.

Harper and Hess \cite{HH} show that this tower is weakly equivalent to the following tower of functors from $\cO \alg(\Ch_k)$ to itself
$$\Id \simeq \cO \overset{\LL}{\circ}_{\cO} (-) \to \cdots  \to \cO^{\le n} \overset{\LL}{\circ}_{\cO} (-) \to \cdots \to \cO^{\le 2} \overset{\LL}{\circ}_{\cO} (-) \to \cO^{\le 1} \overset{\LL}{\circ}_{\cO} (-)$$

Thus, $\cO$-algebras complete with respect to the Goodwillie tower of the identity functor are equivalent to $\cO$-algebras that are complete with respect to the arity cofiltration on the operad $\cO$. Such algebras are called homotopy pro-nilpotent $\cO$-algebras.  These are $\cO$-algebras that arise as the inverse limit of homotopy nilpotent $\cO$-algebras.

Ching and Harper \cite{CH} (following \cite{FG}) show that connected homotopy pro-nilpotent $\cO$-algebras are equivalent to connected, ind-nilpotent $\cO^{\vee}$-coalgebras. Thus, connected $\cO$-algebras that are complete with respect to the Goodwillie tower of the identity functor are equivalent to connected, ind-nilpotent $\cO^{\vee}$-coalgebras. 

Summarizing, this is the relationship between the arity cofiltration on $\cO$ and the Koszul dual cooperad $\cO^{\vee}$. The arity cofiltration of $\cO$ agrees with the Goodwillie tower of $\Id: \cO \alg(\Ch_k) \to \cO \alg(\Ch_k)$ and there is a fundamental relationship between the Goodwillie tower of $\Id: \cO \alg(\Ch_k) \to \cO \alg(\Ch_k)$ and the Koszul duality between the operad $\cO$ and the cooperad $\cO^{\vee}$.

Comparable to \cite{CH}, our main theorem shows that connected associative (and $\cE_n$) algebras complete with respect to the commutator cofiltration are equivalent to connected $K$ (and $K_n$) coalgebras in $Com \alg(\Ch_k)$. Thus, the commutator cofiltration is a relative analogue of the Goodwillie tower of the identity functor.

\subsection{Acknowledgements} The author would like to thank David Ben-Zvi and Andrew Blumberg for all of their help and guidance. We thank John Harper for technical conversations and for introducing the author to the notion of uniformity. We also thank Owen Gwilliam, Joey Hirsh, Hendrik Orem, Nick Rozenblyum, Pavel Safronov, and Travis Schedler for helpful conversations.

\section{Review of Model Categories of Operadic Algebras in Symmetric Spectra} \label{Model Categories}
In this section we recall the basic definitions of symmetric spectra of \cite{HSS} and \cite{S}.  We use the model category of symmetric spectra as a model for the homotopy theory of spectra for its well-behaved properties.

\begin{defn} The category $\Sp^{\Sigma}$ of symmetric spectra is the category where
\begin{itemize}
\item objects $X$ are sequences $X_n \in \sSet_*^{\Sigma_n}$ together with structure maps
$$S^m \wedge X_n \to X_{m+n}$$
\item morphisms $f:X \to Y$ are sequences of maps $f_n: X_n \to Y_n$ compatible with the structure maps and the $\Sigma_n$ actions. 
\end{itemize}
There is a symmetric monoidal structure on $\Sp^{\Sigma}$ with unit, $\SS$, the sphere spectrum. 
\end{defn}

\begin{defn} A symmetric spectrum $X$ is
\begin{itemize}
\item an $\Omega$-spectrum if for each $n \ge 0$, the simplicial set $X_n$ is a Kan complex and \\ $\Map_{\sSet_*}(S^1, X_n) \to X_{n+1}$ is a weak equivalence of simplicial sets.
\item injective if it has the extension property with respect to every monomorphism of symmetric spectra that is a level equivalence.
\end{itemize}
\end{defn}

\begin{defn} A map $f:X \to Y$ of symmetric spectra is a stable equivalence if for every injective $\Omega$-spectrum $Z$ the induced map
$[Y,Z] \to [X ,Z]$ on homotopy classes of spectrum morphisms is a bijection.
\end{defn}

Let $k$ be a field of characteristic 0, and let $Hk$ be the corresponding commutative algebra in the symmetric monoidal category of symmetric spectra, $\Sp^{\Sigma}$. We can consider the category of $Hk$-modules in $\Sp^{\Sigma}$, which we will denote $\Mod_{Hk}$. In the sequel, we will use the standard $t$-structure on $\Mod_{Hk}$ where the truncations correspond to Postnikov towers.

\begin{nota} For $m \ge 0$ and $i: H \subseteq \Sigma_m$, we define $G^H_m: \sSet_* \to \sSet_*^{\Sigma}$ to be the left adjoint to
\[\xymatrix{
\Sp^{\Sigma} \ar[r]^-{Ev_m} & \sSet_*^{\Sigma_m} \ar[r]^-{i^*}  & \sSet_*^H \ar[r]^-{\underset{H}{\lim}} & \sSet_*
}\]

Thus, $G^H_m(K)$ is the symmetric spectrum generated from $\Sigma_m \cdot_{H} K \in \sSet_*^{\Sigma_m}$.
\end{nota}

\begin{thm} The positive flat stable model structure on $\Mod_{Hk}$ has weak equivalences the stable equivalences and cofibrations the retracts of (possibly transfinite) compositions of pushouts of maps 
$$Hk \wedge G^H_m(\partial \Delta[k]_+) \to Hk \wedge G^H_m( \Delta[k]_+) \quad (m \ge 1, H \subseteq \Sigma_m, k \ge 0)$$
and fibrations the maps with the right lifting property with respect to acyclic cofibrations.
\end{thm}

Let $\cO$ be an operad in $\Mod_{Hk}$, and let $\cO \alg(\Mod_{Hk})$ be the category of $\cO$-algebras in $\Mod_{Hk}$.

\begin{rem} In this thesis, we only consider operads $\cO$ starting in arity $1$ and with $\cO(1) \simeq Hk$. That is, we only consider reduced operads in $\Mod_{Hk}$.
\end{rem}

\begin{thm} The positive flat stable model structure on $\cO \alg(\Mod_{Hk})$ has weak equivalences the stable equivalences and fibrations the maps that are positive flat stable fibrations in the underlying category $\Mod_{Hk}$. 
\end{thm}

\begin{rem} The simplicial model categories described above are cofibrantly generated in which the generating cofibrations and acyclic cofibrations have small domains.  Thus, we have a simplicial cofibrant replacement comonad $$(Q,q:Q \to 1, \delta:Q \to Q^2)$$ and a simplicial fibrant replacement monad $$(R,r:1 \to R,\mu:R^2 \to R)$$ on each of these model categories \cite[6.1]{BR}.  This technology will be important when we discuss homotopical monads and comonads arising from Quillen adjunctions between various categories of operadic algebras.
\end{rem}

We will focus on the following examples of operads.
\begin{defn} 
We define the following three operads in $\Mod_{Hk}$:
\begin{itemize} 
\item The associative operad, $Ass$, is defined by $Ass(m) := {\Sigma_m}_+ \wedge Hk$.
\item The commutative operad, $Com$, is defined by $Com(m) := {*}_+ \wedge Hk \simeq Hk$.
\item The $\cE_n$ operad is defined by $\cE_n(m) := \cE_{n,\Top_*}(m) \wedge Hk$ where $\cE_{n,\Top_*}$ denotes the usual $\cE_n$ operad in the category of pointed topological spaces.
\end{itemize}
\end{defn}

\begin{rem} We will use the convenient fact that for each arity $r \ge 1$ of each of the three operads $\cO$ above, the map $1(r) \to \cO(r)$ is a positive flat stable cofibration between positive flat stable cofibrant objects in $\Mod_{Hk}$. This is the cofibrancy condition in \cite[1.15]{HH}. In particular, if $\cO$ satisfies the cofibrancy condition and $f:\cO \to \cO'$ is a map of operads,
then $$|\Barc^{\bullet} (\cO',\cO,QX)| \simeq \cO' \overset{\LL}{\circ}_{\cO} X$$
by \cite[1.10]{H}.
\end{rem}

\section{Derived Abelianizations and Relative Koszul Duality}
Given an augmented operad $\cO \to 1$, Koszul duality for non-unital $\cO$-algebras is often expressed via the Quillen adjunction
\[\xymatrix{
\Mod_{Hk} \simeq 1 \alg(\Mod_{Hk}) \ar[r]_-{\Res} & \cO \alg(\Mod_{Hk}) \ar@/_1pc/[l]_{1 \circ_{\cO}(-)}
}\]

Furthermore, $L_{k|A}[-1]$, the shifted relative cotangent complex of an augmented $\cO$-algebra $A \to k$ may be computed via $1 \overset{\LL}{\circ}_{\cO} (\bar{A})$, the left derived functor of indecomposables of the non-unital algebra $\bar{A} := \hofib(A \to k)$.

This is also called the derived cotangent space to $\cO$-algebra $A$ at the point $k$.  If $X$ is a non-unital $\cO$-algebra, then $1 \overset{\LL}{\circ}_{\cO}(X)$ carries a coaction of  $\cO^{\vee} := 1 \overset{\LL}{\circ}_{\cO} 1$, the Koszul dual cooperad to the operad $\cO$.  Thus, an $Hk$-module spectrum that arises via the derived functor of indecomposables of an $\cO$-algebra has an extra structure encoded by the homotopical coaction of $\cO^{\vee}$. 

In this section, we study an analogous set up where the augmented operad $\cO \to 1$ is replaced by the map of operads $Ass \to Com$, and thus may be interpreted as a type of relative Koszul duality.  Furthermore, the derived functor of indecomposables is replaced by the derived functor of abelianization and   the role of the Koszul dual cooperad $\cO^{\vee}$ is replaced by a homotopical comonad which we call $K$. The comonad $K$ may be interpreted as $Com \overset{\LL}{\circ}_{Ass} Com$.  Thus, a commutative algebra that arises via the derived abelianization of an associative algebra has an extra structure encoded by the homotopical coaction of $K$.

\subsection{The Abelianization-Restriction Quillen Adjunction}

The canonical map of operads $Ass \to Com$ induces a restriction functor $$\Res: Com \alg(\Mod_{Hk}) \to Ass \alg(\Mod_{Hk}).$$This restriction functor takes an algebra with a commutative and associative multiplication, and remembers that the multiplication was associative. This restriction functor has a left adjoint, $\Ab$, which may be computed via the relative composition product $\Ab(-) := Com \circ_{Ass}(-)$ by \cite[5.2.12]{LV}. If $X \in Ass \alg(\Mod_{Hk})$, then the relative composition product $Com \circ_{Ass} X$ is defined by the coequalizer
$$Com \circ Ass \circ X \rightrightarrows Com \circ X \to Com \circ_{Ass} X,$$
where the top first arrow is given by the right action of $Ass$ on $Com$ and the bottom first arrow is given by the left action of $Ass$ on $X$.

Furthermore, since the restriction functor preserves fibrations and acyclic fibrations there is a Quillen adjunction: \\
\[\xymatrix{
Com \alg(\Mod_{Hk}) \ar[r]_{\Res} & Ass \alg(\Mod_{Hk}) \ar@/_1pc/[l]_{\Ab}.
}\]

\begin{defn} The derived abelianization of $X \in Ass \alg(\Mod_{Hk})$ is the derived functor of abelianization applied to $X$, $\Ab(QX)$.
Here 
$$Q: Ass \alg(\Mod_{Hk}) \to Ass \alg(\Mod_{Hk})$$
is a cofibrant replacement functor with the additional structure of a comonad.
\end{defn}

\begin{rem} 
If $N$ is a left $Ass$-module in symmetric sequences, the functor 
$$Com \circ_{Ass}-: N \mapsto Com \circ_{Ass} N$$ has a homotopy invariant left derived functor
$Com \overset{\LL}{\circ}_{Ass} -$ which may be computed by $Com \circ_{Ass} R_N$, where $R_N$ is a cofibrant resolution of the object N in the category of left $Ass$-modules.  Moreover, see \cite[1.10]{H} for a proof that the derived functor of abelianization of the associative algebra $X$ is computed via the derived relative composition product
$$\Ab(QX) \simeq Com \overset{\LL}{\circ}_{Ass}(X).$$
\end{rem}

\begin{ex} Let $V \in \Mod_{Hk}$ be cofibrant, and let $\text{T}^{*}(V)$ be the free associative algebra on $V$.  Then, $\text{T}^{*}(V)$ is cofibrant in $Ass \alg(\Mod_{Hk})$ and 
$$Com \overset{\LL}{\circ}_{Ass}(\text{T}^{*}V) \simeq \Sym^{*}(V)$$
is the free commutative algebra on $V$.  
\end{ex}

\begin{rem}
The example above provides a recipe for computing the derived abelianization of any associative algebra in $\Mod_{Hk}$ up to weak equivalence.
Namely, let $T_{Ass}$ be the homotopical monad arising from the adjunction
\[\xymatrix{
Ass \alg(\Mod_{Hk}) \ar[r]_-{\Oblv} & \Mod_{Hk} \ar@/_1pc/[l]_{\Free}.
}\]
If $X$ is an associative algebra and $QX$ a cofibrant replacement of $X$, let $T_{Ass}^{\bullet}QX$ be the simplicial object arising from the monad $T_{Ass}$ acting on $QX$. 
Note that $T_{Ass}^{\bullet}QX$ is a diagram of free associative algebras.
Let $|T_{Ass}^{\bullet}QX|$ be its geometric realization.  
By definition, there is a weak equivalence
$$|T_{Ass}^{\bullet}QX| \simeq QX.$$
Since $\Ab$ is a left adjoint we have
$$|\Ab T_{Ass}^{\bullet}QX| \simeq \Ab |T_{Ass}^{\bullet}QX| \simeq \Ab QX.$$
In summary, the derived abelianization of any associative algebra can be constructed from the derived abelianization of free associative algebras.
\end{rem}

\begin{ex} Consider the algebra of dual numbers $A = k[\epsilon]/\epsilon^2$ with $\deg(\epsilon)=0$ as a differential graded associative algebra concentrated in degree $0$.  Let us represent the algebra of dual numbers as $A = k \oplus V$, where $\dim(V)=1$. Let $A^{\mbox{!`}}$ denote its dual coassociative coalgebra. We have that $A^{\mbox{!`}} = T_*(V[1])$, the cofree coalgebra on $V[1]$.  Let $\overline{T}_*(V[1])$ denote the reduced cofree coalgebra on $V[1]$.

Now let the pair $\{ \Cobar(A^{\mbox{!`}}),d_{\Cobar} \}$ denote the standard Cobar construction on $A^{\mbox{!`}}$, where $d_{\Cobar}$ is the usual Cobar construction differential.  Then $\{ \Cobar(A^{\mbox{!`}}),d_{\Cobar} \}$ is a resolution of $A$. In the example $A = k \oplus V$, we get $\{ T^*(\overline{T}_*(V[1])[-1]),d_{\Cobar} \}$.

Working through the definitions, we obtain that $\{ T^*(\overline{T}_*(V[1])[-1]),d_{\Cobar} \}$ is 
$$k \langle t_1,t_2,\ldots,t_n,\ldots \rangle \thickspace \mbox{where} \thickspace \deg(t_n) = 1-n$$
and 
$$d(t_n) = \sum_{i+j=n}(-1)^{(-i)}t_it_j$$

The derived abelianization of this algebra is
$$k[ t_1,t_2,\ldots,t_n,\ldots] \thickspace \mbox{where} \thickspace \deg(t_n) = 1-n$$
and 
$$d(t_n) = \sum_{i+j=n}(-1)^{(-i)}t_it_j$$
\end{ex}

\begin{ex} The derived abelianization of the commutative algebra $\Sym^{*}(V)$ considered as an associative algebra is 
$$\Sym^*(\overline{\Sym}_*(V[1])[-1]).$$
Here, $\overline{\Sym}_*(V[1])$ denotes the reduced symmetric coalgebra on $V[1]$. Its differential is trivial.

For example, if $V$ is two-dimensional and concentrated in degree 0, then
$$\overline{\Sym}_*(V[1]) = V[1] \oplus \wedge^2(V)[2]$$
Thus, $\Sym^*(\overline{\Sym}_*(V[1])[-1])$ is the free commutative algebra on two generators in degree 0 and one generator in degree -1.
\end{ex}

\begin{warn} If $A \in Com \alg (\Mod_{Hk})$ is restricted to $Ass \alg (\Mod_{Hk})$, then
$$Com \overset{\LL}{\circ}_{Ass}(A) \not\simeq A.$$
However, such an equivalence would hold in an underived setting. Namely, the abelianization of a discrete commutative algebra is itself.
\end{warn}

\subsection{The Homotopical Comonad K} \label{Comonad K}
In this section we define a point-set homotopical comonad $K$ from the data of the abelianization-restriction Quillen adjunction. We also define homotopy coalgebras over $K$ and construct a simplicial $A_{\oo}$-category of homotopy $K$-coalgebras. The constructions and results in this section are entirely formal, i.e. work for any Quillen adjunction between cofibrantly generated simplicial model categories.

The model categories used in the above abelianization-restriction Quillen adjunction are cofibrantly generated. Thus, we have a cofibrant replacement comonad 
$$(Q,q:Q \to 1, \delta:Q \to Q^2)$$
and a fibrant replacement monad 
$$(R,r:1 \to R,\mu:R^2 \to R)$$
on each of these model categories by \cite{BR}.  We may use the technology developed in \cite{BR} to give the compositions $\Ab Q \Res R$ and $\Res R \Ab Q$ the structures of a comonad and a monad ``up to homotopy." We now describe this structure precisely.   

\begin{defn} Let $\iota :Q \to Q \Res R \Ab Q$ and $\pi: R \Ab Q \Res R \to R$ be defined by
$$\iota:Q \to Q^2 \to Q \Res \Ab Q \to Q \Res R \Ab Q $$
$$\pi: R \Ab Q \Res R \to R \Ab \Res R \to R^2 \to R$$
\end{defn}

Then by \cite[4.2]{BR} we have
\begin{prop} The triples $(\Res R \Ab Q,\iota,\Res \pi)$ and $(\Ab Q \Res R, \pi, \Ab \iota)$ define a point-set level homotopical monad and comonad. More precisely, there is a multiplication and comultiplication
$$\Res R \Ab Q \Res R \Ab Q \overset{\Res \pi}{\to} \Res R \Ab Q \quad \Ab Q \Res R \overset{\Ab \iota}{\to} \Ab Q \Res R \Ab Q \Res R$$
and there are zig-zags
$$1 \leftarrow Q \to Q \Res R \Ab Q \to \Res R \Ab Q \quad \Ab Q \Res R \to R \Ab Q \Res R \to R \leftarrow 1$$
representing the unit and counit maps that make $\Res R \Ab Q$ into a monad in $\Ho(Ass \alg(\Mod_{Hk}))$ and 
$\Ab Q \Res R$ into a comonad in $\Ho(Com \alg(\Mod_{Hk}))$.
\end{prop}

\begin{nota} We write $K: Com \alg(\Mod_{Hk}) \to Com \alg(\Mod_{Hk})$ for the homotopical comonad $\Ab Q \Res R$, and we write
$T: Ass \alg(\Mod_{Hk}) \to Ass \alg(\Mod_{Hk})$ for the homotopical monad $\Res R \Ab Q$.
\end{nota}

\begin{rem} If $Y \in Com \alg(\Mod_{Hk})$, then 
$$KY \simeq \Ab Q \Res R Y \simeq Com \overset{\LL}{\circ}_{Ass}(\Res R Y) \simeq (Com \overset{\LL}{\circ}_{Ass}Com) \circ_{Com} (RY).$$
We will refer to the expression
$Com \overset{\LL}{\circ}_{Ass}Com$ as the Koszul dual comonad to the operad $Ass$ relative to $Com$. 
\end{rem}

\begin{defn} A homotopy $K$-coalgebra is a commutative algebra $Y$, together with a map $h: Y \to KY$ so that 
$\pi \cdot Rh = \id$ and $Kh \cdot h = \Ab \iota \cdot h$. That is, the following two diagrams commute
\[\xymatrix{
RY \ar[r]^{Rh} \ar[dr]_{\id} & RKY \ar[d]^{\pi} & & Y \ar[r]^{h} \ar[d]_{h} & KY \ar[d]^{Kh}\\
& RY & & KY \ar[r]^{\Ab \iota} & K^2Y.
}\]
\end{defn}

The coaction of the homotopical comonad $K$ is precisely the extra structure on a commutative algebra when it arises from the derived abelianization of an associative algebra:
\begin{prop} Let $X \in Ass \alg(\Mod_{Hk})$, then $\Ab(QX)$ is a homotopy $K$-coalgebra.
\end{prop}
\begin{proof} There is a natural map $h:\Ab(QX) \to \Ab Q \Res R \Ab(QX)$ using the counit \\ $\iota: Q \to Q \Res R \Ab Q$. One checks that
$\pi \cdot Rh = \id$ and $Kh \cdot h = F\iota \cdot h$.
\end{proof}

\begin{ex} The derived abelianization of a commutative algebra considered as an associative algebra has the structure of a cofree homotopy $K$-coalgebra. In particular, the derived abelianization of $\Sym^*(V)$,
$$Com \overset{\LL}{\circ}_{Ass} \Sym^*(V) \simeq \Sym^*(\overline{\Sym}_*(V[1])[-1])$$
has the structure of a cofree homotopy $K$-coalgebra.
\end{ex}

We now construct homotopically correct mapping spaces between homotopy $K$-coalgebras, following \cite{AC}.
\begin{defn}
Let $A$ and $A'$ be homotopy $K$-coalgebras in $Com \alg(\Mod_{Hk})$. We define the cosimplicial simplicial set $\Hom^{\bullet}_K(A,A')$ by
$$\Hom^{\bullet}_K(A,A') := \Hom_{Com}(QA,RK^{\bullet}A')$$
where the coface maps $\delta^i:  \Hom_{Com}(QA,RK^{m}A') \to  \Hom_{Com}(QA,RK^{m+1}A')$ are defined in a natural way using the $K$-coalgebra structures on $A$ and $A'$ and the comonad comultiplication $RK \to RK^2$. Furthermore, we define the simplicial set of homotopy $K$-coalgebra maps from $A$ to $A'$ as 
$$\widetilde{\Hom}_K(A,A') := \widetilde{\Tot} \Hom_{Com}(QA,RK^{\bullet}A').$$
Here, the notation $\widetilde{\Tot}$ denotes a Reedy fibrant replacement of the corresponding cosimplicial object on which one is taking a totalization.
\end{defn}

\begin{defn} A homotopy $K$-coalgebra map between $A$ and $A'$ is a vertex $f$ in the simplicial set $\widetilde{\Hom}_K(A,A')$. Such an $f$ consists of a collection of maps 
$$f_n: \Delta^n \to \Hom_{Com}(QA,RK^{n}A')$$
satisfying some compatibility relations.  In particular, there is an underlying map $f_0:QA \to RA'$ and the map $f_1: \Delta^1 \to \Hom(QA,RKA')$ gives a homotopy between the two composites in the square
\[\xymatrix{
QA \ar[r] \ar[d] & RKA \ar[d] \\
QA' \ar[r] & RKA'
}\]
The maps $f_n$ can be viewed as forming a set of higher coherent homotopies that generalize this.
\end{defn}

\begin{prop} The construction above determines a simplicial $A_{\oo}$-category whose objects are homotopy $K$-coalgebras and whose morphism spaces are the simplicial sets $\widetilde{\Hom}_K(A,A')$. We denote this simplicial $A_{\oo}$-category by $$K \coalg(Com \alg(\Mod_{Hk})).$$
\end{prop}
\begin{proof} This is similar to \cite[1.14]{AC}.
The fat totalization of a cosimplicial object is the simplicial object defined by the end, $\Tot^{\fat}(X^{\bullet}) := \Hom_{\Delta_{\inj}}(\Delta^{\bullet},X^{\bullet})$, were $\Delta_{\inj}$ is the subcategory of the simplicial indexing category consisting only of injective maps.  We use the fat totalization in this proof as a model for computing the simplicial set $\widetilde{\Hom}_K(A,A')$, where $A$ and $A'$ are homotopy $K$-coalgebras. This is a homotopically correct construction since $\Hom^{\bullet}_K(A,A')$ is an objectwise fibrant cosimplicial diagram.

Now, we define a model for the $A_{\oo}$-operad \cite[3.5]{MS} by
$$A_n := \Hom_{\Delta_{\inj}}(\Delta^{\bullet},(\Delta^{\bullet})^{\Box n})$$
where $\Box$ denotes the Batanin symmetric monoidal structure on cosimplicial simplicial sets. 

For $n \ge 0$ and for homotopy $K$-coalgebras $Y_0, \ldots, Y_n$, we define natural composition maps
$$A_n \times \widetilde{\Hom}_K(Y_0,Y_1) \times \ldots \times \widetilde{\Hom}_K(Y_{n-1},Y_n) \to \widetilde{\Hom}_K(Y_0,Y_n)$$
by the composites
\begin{align*}
\Hom_{\Delta_{\inj}}(\Delta^{\bullet},(\Delta^{\bullet})^{\Box n}) &\times \Hom_{\Delta_{\inj}}(\Delta^{\bullet}, \Hom^{\bullet}_K(Y_0,Y_1)) \times \ldots \times \Hom_{\Delta_{\inj}}(\Delta^{\bullet}, \Hom^{\bullet}_K(Y_{n-1},Y_n)) \\
 & \to \Hom_{\Delta_{\inj}}(\Delta^{\bullet},(\Delta^{\bullet})^{\Box n}) \times \Hom_{\Delta_{\inj}}(\Delta^{\bullet},\Hom^{\bullet}_K(Y_0,Y_1) \Box \ldots \Box \Hom^{\bullet}_K(Y_{n-1},Y_n)) \\
& \to  \Hom_{\Delta_{\inj}}(\Delta^{\bullet}, \Hom^{\bullet}_K(Y_0,Y_1) \Box \ldots \Box \Hom^{\bullet}_K(Y_{n-1},Y_n)) \\
& \to  \Hom_{\Delta_{\inj}}(\Delta^{\bullet}, \Hom^{\bullet}_K(Y_0,Y_n))
\end{align*}
The result follows.
\end{proof}

\begin{defn} The homotopy category of homotopy $K$-coalgebras has as objects homotopy $K$-coalgebras and morphisms from $A$ to $A'$ are the path components of the simplicial mapping space. We denote this set of morphisms by
$$[A,A']_K := \pi_0 \widetilde{\Hom}_K(A,A').$$
\end{defn}

\subsection{The Cobar Construction for Homotopy $K$-Coalgebras}
In this section we define the cobar construction for a homotopy $K$-coalgebra and define the so-called $\Ab$-completion of an associative algebra based on this cobar construction.  We then show that $\Ab$-complete associative algebras embed into homotopy $K$-coalgebras.  We also define the derived unit and derived counit maps for the adjunction 
$$Ass \alg (\Mod_{Hk}) \leftrightarrows K \coalg(Com \alg(\Mod_{Hk}))$$
The constructions and results in this section are entirely formal, i.e. work for any Quillen adjunction between cofibrantly generated simplicial model categories.

\begin{defn} Let $Y$ be a homotopy $K$-coalgebra.  The cobar construction on $Y$, $\Cobar^{\bullet}_K(Y)$, is the cosimplicial object in $Ass \alg(\Mod_{Hk})$ given by
\[\xymatrix{
Q \Res RY \ar@<-.5ex>[r] \ar@<.5ex>[r] & Q \Res R K Y \ar@<1.0ex>[r] \ar@<0.0ex>[r] \ar@<-1.0ex>[r] & Q \Res R K^2 Y  \cdots
}\]
\end{defn}

\begin{defn} The $\Ab$-completion of an associative algebra $X$ is the $\Tot$ of a Reedy fibrant replacement of the cobar construction on the homotopy $K$-coalgebra $\Ab(QX)$.
That is, $\widetilde{\Tot}$ of the cosimplicial object
\[\xymatrix{
Q\Res R \Ab (QX) \ar@<-.5ex>[r] \ar@<.5ex>[r] & QK \Res R \Ab (QX) \ar@<1.0ex>[r] \ar@<0.0ex>[r] \ar@<-1.0ex>[r] & QK^2\ \Res R \Ab (QX)  \cdots
}\]
or of the equivalent cosimplicial object,
\[\xymatrix{
QT(X) \ar@<-.5ex>[r] \ar@<.5ex>[r] & QT^2(X) \ar@<1.0ex>[r] \ar@<0.0ex>[r] \ar@<-1.0ex>[r] & QT^3(X)  \cdots
}\]
We denote the $\Ab$-completion of $X$ by $\hat{X}_{\Ab}$. 
\end{defn}

By construction, the $\Ab$-completion is functorial and preserves weak-equivalences:
\begin{prop} The derived completion $X \mapsto \hat{X}_{\Ab}$ defines a functor 
$$\hat{(-)}_{\Ab}:Ass \alg(\Mod_{Hk}) \to Ass \alg(\Mod_{Hk})$$ and if 
$X \to Y$ is a weak equivalence in $Ass \alg(\Mod_{Hk})$, then $\hat{X}_{\Ab} \to \hat{Y}_{\Ab}$ is also a weak equivalence.
\end{prop}

\begin{defn} Let $X \in Ass \alg (\Mod_{Hk})$. The derived unit map of the 
$$Ass \alg (\Mod_{Hk}) \leftrightarrows K \coalg(Com \alg(\Mod_{Hk}))$$
adjunction is the natural map
$$\eta: QX \to \widetilde{\Tot}(Q \Res R K^{\bullet} \Ab(QX))$$
which is induced from the canonical map $QX \to Q \Res R \Ab(QX)$.
\end{defn}

\begin{defn} We say that $X$ is $\Ab$-complete if the derived unit map $$\eta: QX \to \widetilde{\Tot}(Q \Res R K^{\bullet} \Ab(QX))$$ is a weak equivalence. We denote the full subcategory 
of $\Ab$-complete algebras by $Ass^{\abcompl} \alg(\Mod_{Hk})$.
\end{defn}

We now show that the functor of derived abelianzation is homotopically fully-faithful when restricted to $\Ab$-complete associative algebras.
This is a version of \cite[2.15]{AC}.
\begin{prop} For $X,X' \in Ass \alg(\Mod_{Hk})$ with $X$ cofibrant and $X'$ fibrant and $\Ab$-complete, the natural map
$$\Hom_{Ass \alg(\Mod_{Hk})}(X,X') \to \widetilde{\Hom}_K(\Ab(QX),\Ab(QX')) $$
is a weak equivalence of simplicial sets.
\end{prop}
\begin{proof} 
Consider the following diagram of simplicial sets
\smallskip
\newline
\resizebox{14cm}{!}{\begin{xy} \newline
\xymatrix{
\Hom_{Ass}(X,QX') \ar[r]^-{\eta} \ar[dd]_-{\epsilon} \ar[dr]^-{\Ab Q} & \widetilde{\Tot}\Hom_{Ass}(X,Q \Res R K^{\bullet} \Ab(QX'))\ar[dr]^-{\Ab Q} \\
& \Hom_K(\Ab Q X,\Ab Q Q X')  \ar[r]^-{\eta} & \widetilde{\Tot}\Hom_K(\Ab Q X, \Ab Q Q \Res R K^{\bullet} \Ab Q X') \\
\Hom_{Ass}(X,X') \ar[dr]^-{\Ab Q} \\
& \Hom_K(\Ab Q X,\Ab Q X') \ar[r]^-{\eta} \ar[uu]^-{\delta} &  \widetilde{\Tot}\Hom_K(\Ab Q X, \Ab Q \Res R K^{\bullet} \Ab Q X') \ar[uu]^-{\delta}
} 
\end{xy}}
\smallskip
\newline
We wish to show that the composite of the bottom two maps in the diagram
$$\Hom_{Ass}(X,X') \to \Hom_K(\Ab Q X,\Ab Q X') \to  \widetilde{\Tot}\Hom_K(\Ab Q X, K^{\bullet+1} \Ab Q X')$$
is weak equivalence. Notice the right-hand object is isomorphic to $$\widetilde{\Hom}_K(\Ab Q X, \Ab Q X').$$
It suffices to check that all of the `outer' maps in the diagram above are weak equivalences.  But this is easy to check. For instance, the top $\eta$ in the diagram is a weak equivalence since $X'$ is assumed to be $\Ab$-complete. The top right map labelled $\Ab Q$ is also a weak equivalence since for any $Y \in Com \alg(\Mod_{Hk})$ the map
$$\Hom_{Ass}(X,Q \Res R Y) \overset{\Ab Q}{\to} \Hom_K(\Ab Q X,\Ab Q Q \Res R Y)$$
is a weak equivalence. This follows from the commutative diagram:
\[\xymatrix@C=1.1em{
\Hom_{Ass}(X,Q \Res R Y) \ar[r]^-{\Ab Q}  \ar[d]_{\simeq}^-{\epsilon} & \Hom_K(\Ab Q X, \Ab Q Q \Res R Y) \ar[d]_{\simeq}^-{\epsilon}\\
\Hom_{Ass}(X,\Res R Y)  \ar[r]^-{\Ab Q} \ar[drr]^-{\simeq} &  \Hom_K(\Ab Q X, \Ab Q \Res R Y) \ar[r]^-{\simeq} & \Hom_{Com}(\Ab Q X, R Y) \ar[d]^-{\simeq} \\
& & \Hom_{Ass}(Q X, \Res R Y)
}\]
\end{proof}

\begin{cor} The functor $\Ab(Q-)$ determines a fully-faithful embedding of the homotopy category of $\Ab$-complete associative algebras to the homotopy category of homotopy $K$-coalgebras.
\end{cor}

We now define the derived counit map corresponding to the 
$$Ass \alg (\Mod_{Hk}) \leftrightarrows K \coalg(Com \alg(\Mod_{Hk}))$$
adjunction.
\begin{defn} Let $Y \in K \coalg(Com \alg(\Mod_{Hk}))$, we define a map of homotopy $K$-coalgebras 
$$\epsilon: \Ab Q \widetilde{\Tot}[\Cobar_K^{\bullet}(Y)] \to Y$$
as follows.  This consists of simplicial sets
$$\epsilon_n:\Delta^n \to \Hom_{Com}(\Ab Q \widetilde{\Tot}[Q \Res R K^{\bullet}Y], R K^nY)$$
that are the composites
\begin{align*}
\Delta^n &\to \Hom_{Ass}(\widetilde{\Tot}[Q \Res R K^{\bullet}Y],Q \Res R K^n Y) \\
&\to \Hom_{Com}(\Ab Q \widetilde{\Tot}[Q \Res R K^{\bullet}Y],\Ab Q Q \Res R K^n Y) \\
&\to \Hom_{Com}(\Ab Q \widetilde{\Tot}[Q \Res R K^{\bullet}Y], R K^n Y)
\end{align*}
where the first map is adjoint to the natural projection in $Ass \alg(\Mod_{Hk})$
$$\widetilde{\Tot}[Q \Res R K^{\bullet}Y] \to \Hom(\Delta^n,Q \Res R K^n Y)$$
the second map uses the simplicial enrichment of $\Ab Q$, and the third map involves the counit maps
$$\Ab QQ \Res R \to \Ab \Res R \to R.$$
\end{defn}

In the main theorem of this chapter, we will show the derived unit and derive counit map induce equivalences under certain connectivity restraints.
\section{Derived NC-complete Algebras}
In this section we introduce the derived NC-(co)filtration for associative algebras in $\Mod_{Hk}$, and use this (co)filtration to define derived NC-complete algebras in $\Mod_{Hk}$.
Thus, this section produces a derived version of Kapranov's theory of NC-complete algebras \cite{K}.  
  
\begin{defn} There is a cofiltration on the operad $Ass$:
$$Ass \rightarrow \ldots \rightarrow  Ass^{\le n} \rightarrow \ldots \rightarrow Ass^{\le 2} \rightarrow Ass^{\le 1} \rightarrow Ass^{\le 0} \simeq Com,$$
where $Ass^{\le n}$ consists of associative operations with less than or equal to n appearances of the operation corresponding to the commutator bracket. 
We will refer to this as the commutator cofiltration (or the NC-cofiltration) on the associative operad.
\end{defn}

\begin{rem} The descending filtration, $F_i(Ass) := Ass^{\ge i}$, corresponding to this cofiltration is precisely the filtration on the associative operad whose associated graded operad is the Poisson operad $\cP_1$. More explicitly, the Poisson operad is naturally graded by how many Poisson brackets appear in a given operation, and this grading arises as 
$$\cP_1 = \gr_F(Ass) := F_0(Ass)/F_1(Ass) \oplus F_1(Ass)/F_2(Ass) \oplus F_2(Ass)/F_3(Ass) \oplus \cdots $$
Furthermore, since $Ass \to Com$ is induced by a map of operads in $k$-vector spaces, it has a well-defined kernel $I \subseteq Ass$.  The descending filtration $F_i(Ass)$ is simply the $I$-adic filtration on $Ass$. In particular, there is a well-defined notion of powers of the ideal $I$ since we are in a discrete setting. A derived version of the construction of powers of an operadic ideal for operads valued in the model category spectra or the model category of chain complexes is desirable. A first step to a derived version of powers of an operadic ideal is discussed in \cite[2.4]{MV} using the language of a cotangent complex formalism for operads.
\end{rem}

\begin{nota} We will also use the notation $Ass^{>k}$ to denote operations with greater than k appearances of the commutator bracket and $Ass^{=k}$ to denote operations with exactly $k$ appearances of the commutator bracket. The notation $Ass^{\ge k+1}$ will be synonymous with $Ass^{>k}$.  In particular, there are fiber sequences of $Ass$-bimodules $$Ass^{=k} \to Ass^{\le k} \to Ass^{\le k-1}.$$
\end{nota}

\begin{rem} In the sequel, we will use the fact that $Ass^{> k}$ starts in arity $k+2$ and $Ass^{=k}$ starts in arity $k+1$. This observation is central for our connectivity estimates.
\end{rem}

\begin{defn} \label{derived NC-cofiltration}
The derived NC-cofiltration on $X \in Ass \alg(\Ch_k)$ is the tower
$$X \to \cdots \to Ass^{\le n} \overset{\LL}{\circ}_{Ass}(X) \rightarrow \cdots \rightarrow Ass^{\le 1} \overset{\LL}{\circ}_{Ass}(X) \rightarrow Com \overset{\LL}{\circ}_{Ass}(X)$$
\end{defn}

\begin{defn} An algebra $X \in Ass \alg(\Mod_{Hk})$ is called derived NC-nilpotent of degree n if 
$$X \simeq \holim [Ass^{\le n} \overset{\LL}{\circ}_{Ass}(X) \rightarrow \ldots \rightarrow Ass^{\le 1} \overset{\LL}{\circ}_{Ass}(X) \rightarrow Com \overset{\LL}{\circ}_{Ass}(X)].$$
\end{defn}

\begin{defn} The derived NC-completion of an algebra $X \in Ass \alg(\Mod_{Hk})$ is 
$$\hat{X}_{NC} := \underset{n}{\holim} [Ass^{\le n} \overset{\LL}{\circ}_{Ass}(X)].$$
\end{defn}

\begin{defn} \label{nccomplete} An algebra $X \in Ass \alg(\Mod_{Hk})$ is called derived NC-complete if there is a weak-equivalence
$$X \simeq \hat{X}_{NC}.$$
We let $Ass \alg^{\nccompl}(\Mod_{Hk})$ denote the full subcategory of $Ass \alg(\Mod_{Hk})$ formed by derived NC-complete associative algebras.
\end{defn}

\begin{ex} We will prove in the sequel that $0$-connected associative algebras are in particular derived NC-complete.
\end{ex}

\begin{rem}  Although there is a weak equivalence of operads $$Ass \simeq  \underset{n}{\holim} (Ass^{\le n}),$$ that is, the associative operad is complete with respect to the NC-filtration, it is a non-trivial condition for a derived associative algebra to be derived NC-complete. This may be compared with the arity cofiltration for an augmented operad $\cO \to 1$ arising from the Goodwillie tower of the identity functor on the category $\cO \alg(\Mod_{Hk})$. One trivially has $\cO \simeq  \underset{n}{\holim} (\cO^{\le n})$, but it is clearly a non-trivial condition for an $\cO$-algebra to be homotopy pro-nilpotent. 
\end{rem}

The following observation inspired the development of this chapter. It may be interpreted as an analogue of Nakayama's Lemma.
\begin{prop} \label{conservative} The functor 
$$\Ab(Q-): Ass \alg^{\nccompl}(\Mod_{Hk}) \to K \coalg(Com \alg(\Mod_{Hk}))$$
is conservative. 
\end{prop}
\begin{proof} Let $f: X \to X'$ be a map of derived NC-complete associative algebras such that 
$$Com \overset{\LL}{\circ}_{Ass}(X) \to Com \overset{\LL}{\circ}_{Ass}(X')$$
 is a weak equivalence. We wish to show this implies $f: X \to X'$ is also a weak equivalence.  Since $X$ and $X'$ are derived NC-complete we have
$$X \simeq \underset{n}{\holim} [Ass^{\le n} \overset{\LL}{\circ}_{Ass}(X)] \quad \text{and} \quad X' \simeq \underset{n}{\holim} [Ass^{\le n} \overset{\LL}{\circ}_{Ass}(X')].$$
We proceed by induction on $n$ using the fiber sequences 
$$Ass^{= n} \overset{\LL}{\circ}_{Ass}(X) \to Ass^{\le n} \overset{\LL}{\circ}_{Ass}(X) \to Ass^{\le n-1} \overset{\LL}{\circ}_{Ass}(X).$$
The right $Ass$-action map $Ass^{= n} \circ Ass \to Ass^{=n}$ factors through the commutative operad: 
\[\xymatrix{
Ass^{= n} \circ Ass \ar[r] \ar[d] & Ass^{=n} \\
 Ass^{=n} \circ Com \ar[ur] 
}\]
Thus, we have weak equivalences
$$Ass^{= n} \overset{\LL}{\circ}_{Ass}(X) \simeq Ass^{= n} \overset{\LL}{\circ}_{Com} (Com \overset{\LL}{\circ}_{Ass}(X))$$
and hence fiber sequences
$$Ass^{= n} \overset{\LL}{\circ}_{Com} (Com \overset{\LL}{\circ}_{Ass}(X)) \to Ass^{\le n} \overset{\LL}{\circ}_{Ass}(X) \to Ass^{\le n-1} \overset{\LL}{\circ}_{Ass}(X).$$
So if 
$$Com \overset{\LL}{\circ}_{Ass}(X) \to Com \overset{\LL}{\circ}_{Ass}(X')$$ 
is a weak equivalence, then
$$Ass^{\le n} \overset{\LL}{\circ}_{Ass}(X) \to Ass^{\le n} \overset{\LL}{\circ}_{Ass}(X')$$ is a weak equivalence for every $n \ge 0$.
The conclusion follows since a map of homotopy $K$-coalgebras induces an isomorphism in the homotopy category of $K$-coalgebras if and only if the underlying map of commutative algebras is an isomorphism in the homotopy category of commutative algebras.
\end{proof}

\begin{rem} Since the associated graded operad of the NC-(co)filtered associative operad is the Poisson operad $\cP_1$, the proposition above implies there is a connection between this cofiltration and relative Koszul duality (i.e. the homotopical comonad $K$). This relationship will be further strengthened by the main theorem.
\end{rem}

We will now show that $0$-connected algebras are derived NC-complete. 
\begin{nota} Let $\Mod^{> 0}_{Hk}$ denote the subcategory of $Hk$-module spectra with homotopy groups concentrated in positive degree.
\end{nota}

The following lemma is an analogue of \cite[4.33]{HH}.
\begin{lem} \label{connectivity estimate} If $X \in Ass \alg(\Mod^{> 0}_{Hk})$, then
\begin{itemize}
\item $Ass^{\le k} \overset{\LL}{\circ}_{Ass}(X)$ is 0-connected.
\item $Ass^{= k} \overset{\LL}{\circ}_{Ass}(X)$ is $k$-connected.
\item $Ass^{> k} \overset{\LL}{\circ}_{Ass}(X)$ is $(k+1)$-connected.
 \end{itemize}
\end{lem}

\begin{proof} 
First recall the following two facts. 
\begin{itemize}
\item Geometric realizations of simplicial $Hk$-modules that are $k$-connected in each degree are $k$-connected. 
\item If $X,Y \in \Mod_{Hk}$ are $n$-connected and $m$-connected,
then $X \wedge_{Hk} Y$ is \\ $(m+n+1)$-connected.
\end{itemize}  
Now, use the observation that $Ass^{=k}$ starts in arity $k+1$ and and $Ass^{>k}$
starts in arity $k+2$, and that the derived relative composition products 
$$Ass^{\le k} \overset{\LL}{\circ}_{Ass}(X), \quad Ass^{= k} \overset{\LL}{\circ}_{Ass}(X), \quad \text{and} \quad Ass^{> k} \overset{\LL}{\circ}_{Ass}(X)$$
can be computed using the bar constructions
$$|\Barc^{\bullet}(Ass^{\le k},Ass,QX)|, \quad |\Barc^{\bullet}(Ass^{= k},Ass,QX)|,  \quad \text{and} \quad |\Barc^{\bullet}(Ass^{> k},Ass,QX)|.$$
The result follows.
\end{proof}

Recall the definition of ${\lim}^1$ as presented in \cite{M}.
\begin{defn}
Given a sequence of homomorphisms of abelian groups $$\{f_i:A_{i+1} \to A_i, i \ge 1\},$$
there is an exact sequence
$$0 \to \lim A_i \overset{\beta}{\to} \prod_i A_i \overset{\alpha}{\to} {\lim}^1A_i \to 0$$
where $\alpha$ is the difference of the identity map and the map with coordinates $f_i$ and $\beta$ is the map whose projection to $A_i$ is the 
canonical map given by the definition of a limit. Thus, we define ${\lim}^1A_i$ to be the cokernel above.
\end{defn}

Also recall \cite[4.33]{HH}, which is based on \cite[IX]{BK}.
\begin{lem} Given a tower $B_{\bullet}$ in $\Mod_{Hk}$
$$\ldots \to B_3 \to B_2 \to B_1 \to B_0$$ 
There are natural short exact sequences
$$0 \to {\lim}^1(\pi_{i+1}B_{\bullet}) \to \pi_i \holim (B_{\bullet}) \to \lim (\pi_i B_{\bullet}) \to 0.$$
\end{lem}

We are now ready to prove the main result of this section.
\begin{thm} \label{connected is complete} There is an equivalence 
$$Ass \alg^{\nccompl}(\Mod^{> 0}_{Hk}) \simeq Ass \alg(\Mod^{> 0}_{Hk}).$$ 
\end{thm}

\begin{proof} We wish to show that if $X \in  Ass \alg(\Mod^{> 0}_{Hk})$, then 
$$QX \simeq  \underset{k}{\holim} [Ass^{\le k} \overset{\LL}{\circ}_{Ass}(X)].$$
It suffices to show that the natural map
$$\underset{k}{\holim} [Ass \overset{\LL}{\circ}_{Ass}(X)] \to  \underset{k}{\holim} [Ass^{\le k} \overset{\LL}{\circ}_{Ass}(X)]$$
is a weak equivalence.  Here the left hand side is a constant homotopy limit.

Consider the following commutative diagram

\[\xymatrix{
\pi_i \underset{k}{\holim} [Ass \overset{\LL}{\circ}_{Ass}(X)] \ar[r] \ar[d] & \pi_i \underset{k}{\holim} [Ass^{\le k} \overset{\LL}{\circ}_{Ass}(X)] \ar[d] \\
\underset{k}{\holim} [\pi_i (Ass \overset{\LL}{\circ}_{Ass}(X))] \ar[r] & \underset{k}{\holim} [\pi_i (Ass^{\le k} \overset{\LL}{\circ}_{Ass}(X))]
}\]

We wish to show that if $X$ is connected, the top horizontal arrow is a weak equivalence. We will do this by showing the other three arrows in the commutative diagram are weak equivalences. It is easy to see that the left hand vertical map is a weak equivalence: indeed ${\lim}^1\pi_{i+1}|\Barc(Ass,Ass,X)|=0$ since we are working with a constant tower.
The bottom horizontal arrow is a weak equivalence by \ref{connectivity estimate}. Similarly, it follows from \ref{connectivity estimate} that 
$$\pi_i |\Barc(Ass^{\le k+1},Ass,X)| \to \pi_i |\Barc(Ass^{\le k},Ass,X)|$$
 is an isomorphism for $i \le k+1$ and surjective for $i=k+2$. Hence,
$\pi_i |\Barc(Ass^{\le k+1},Ass,X)|$ is eventually constant, so ${\lim}^1\pi_{i+1}|\Barc(Ass^{\le k},Ass,X)|=0$, and thus the right hand vertical arrow is a weak equivalence.  The result follows.
\end{proof}

\subsection{Connectivity Estimates for Derived Abelianizations}

We now analyze how the derived abelianization functor interacts with (highly) connected associative algebras. This may be interpreted as an analogue of 
\cite[1.8]{HH}.
\begin{thm} If $X$ and $Y$ are cofibrant $0$-connected associative algebras, then
\begin{enumerate}[label*=\arabic*.]
\item The associative algebra $X$ is $n$-connected if and only if $Com \overset{\LL}{\circ}_{Ass}(X)$ is $n$-connected.
\item If $Com \overset{\LL}{\circ}_{Ass}(X)$ is $n$-connected, then the map $\pi_k (X) \to \pi_k (Com \overset{\LL}{\circ}_{Ass}(X))$ is an isomorphism for $k \le 2n+1$ and surjective for $k=2n+1$.
\item A map $f:X \to Y$ is $n$-connected if and only if the induced map\\ $Com \overset{\LL}{\circ}_{Ass}(X) \to Com \overset{\LL}{\circ}_{Ass}(Y)$ is $n$-connected.
\end{enumerate}
\end{thm}

\begin{proof} Assume that $Com \overset{\LL}{\circ}_{Ass}(X)$ is $n$-connected. Then $|\Barc^{\bullet}(Com,Ass,QX)|$ is $n$-connected and 
$|\Barc^{\bullet}(Ass^{=k},Ass,QX)|$ is $((k+1)n+k)$-connected for $k \ge 1$. Hence, the maps 
$$\pi_i(Ass^{\le k} \overset{\LL}{\circ}_{Ass}(X)) \to \pi_i(Ass^{\le k-1} \overset{\LL}{\circ}_{Ass}(X))$$
 are isomorphisms for $i \le (k+1)n+k$ and surjections for $i=(k+1)(n+1)$. In particular, for each $i \le 2n+1$ the tower 
$\{ \pi_i(Ass^{\le k} \overset{\LL}{\circ}_{Ass}(X)) \}$ is a tower of isomorphisms, and since $Com \overset{\LL}{\circ}_{Ass}(X)$ is $n$-connected, it follows that each stage of the tower is $n$-connected. Since $X$ is $0$-connected, it follows that 
$$\pi_i(X) \to \pi_i(Ass^{\le k} \overset{\LL}{\circ}_{Ass}(X))$$
 is an isomorphism for $i \le k$. Taking $k$ sufficiently large shows that $X$ is $n$-connected.

Conversely, assume that $X$ is $n$-connected. Then $|\Barc^{\bullet}(Ass^{\le k},Ass,QX)|$ is $n$-connected and both $|\Barc^{\bullet}(Ass^{>k-1},Ass,QX)|$ and
 $|\Barc^{\bullet}(Ass^{=k},Ass,QX)|$ are $((k+1)n+k)$-connected for $k \ge 1$. Thus,  the natural maps 
 $$\pi_i(X) \to \pi_i(Ass^{\le k-1} \overset{\LL}{\circ}_{Ass}(X))$$
 and 
 $$\pi_i(Ass^{\le k} \overset{\LL}{\circ}_{Ass}(X)) \to \pi_i(Ass^{\le k-1} \overset{\LL}{\circ}_{Ass}(X))$$
 are isomorphisms for $i \le (k+1)n+k$ and surjections for $i=(k+1)(n+1)$. Thus, 
 $$\pi_i(X) \to \pi_i(Com \overset{\LL}{\circ}_{Ass}(X))$$
 is an isomorphism for $i \le 2n+1$ and a surjection for $i=2n+2$. Since $X$ is $n$-connected, it follows that $Com \overset{\LL}{\circ}_{Ass}(X)$ is $n$-connected.
 
 Now, let $f: X \to Y$ be a map $Ass$-algebras such that 
 $$Com \overset{\LL}{\circ}_{Ass}(X) \to Com \overset{\LL}{\circ}_{Ass}(Y)$$
 is an $n$-connected map. Consider the commutative diagram
\[\xymatrix{
|\Barc^{\bullet}(Ass^{=k},Ass,X)| \ar[r] \ar[d] & |\Barc^{\bullet}(Ass^{\le k},Ass,X)| \ar[r] \ar[d] & |\Barc^{\bullet}(Ass^{\le k-1},Ass,X)| \ar[d] \\
|\Barc^{\bullet}(Ass^{=k},Ass,Y)| \ar[r] & |\Barc^{\bullet}(Ass^{\le k},Ass,Y)| \ar[r] & |\Barc^{\bullet}(Ass^{\le k-1},Ass,Y)|. 
}\]
Since the left hand and right hand vertical maps are  $n$-connected for $k \ge 2$, the middle vertical map is $n$-connected for $k \ge 2$.  Thus 
$f:X \to Y$ is an $n$-connected map.

Conversely, if $f:X \to Y$ is an $n$-connected map, then 
$$Com \overset{\LL}{\circ}_{Ass}(X) \to Com \overset{\LL}{\circ}_{Ass}(Y)$$
is also an $n$-connected map.
\end{proof}

\begin{cor}
 A map $f:X \to Y$ of $0$-connected associative algebras is a weak equivalence if and only if  the induced map $Com \overset{\LL}{\circ}_{Ass}(X) \to Com \overset{\LL}{\circ}_{Ass}(Y)$ is a weak equivalence.
\end{cor}

\section{Review of Cubical Diagrams in Categories of Operadic Algebras and Blakers-Massey Theorems} \label{cubical}
In the proof of the main theorem, we will need a connectivity estimate for commuting a left adjoint past homotopy limits over $\Delta^{\le n}$.  We will use the language and technology
of cubical diagrams in categories of operadic algebras, Blakers-Massey theorems, and uniformity to prove this connectivity estimate.
We review the relevant definitions and theorems we need. This section contains no new results, and the relevant references are \cite{CH2}, \cite{D}, and \cite{G2}.

\begin{defn} Let W be a finite set and $\cC$ a category.
\begin{itemize}
\item Let $\cP(W)$ be the poset of all subsets of $W$ ordered by inclusion.
\item Let $\cP_0(W) \subset \cP(W)$ be the poset of all non-empty subsets of $W$.
\item Let $\cP_1(W) \subset \cP(W)$ be the poset of all subsets of $W$ not equal to $W$.
\item A $W$-cube in $\cC$ is a functor $X: \cP(W) \to \cC$.  
\end{itemize}
\end{defn}

\begin{nota} If $V \subset W$, we often denote the object in $\cC$ that $X$ assigns to the subset $V \subset W$ by $X_V$.
\end{nota}

\begin{rem}
If $|W|=n$, we say that $X: \cP(W) \to \cC$ is an $n$-cube in $\cC$.  Since $|\cP(W)|=2^n$, 
we equivalently say that an $n$-cube in a category $\cC$ is a functor 
$$\Cube_n \to \cC$$
where $\Cube_n$ is the category with
\begin{itemize}
\item Objects as $n$-tuples $(\alpha_1,\ldots \alpha_n)$, where $\alpha_i \in \{0,1\}$.
\item A single morphism $(\alpha_1,\ldots \alpha_n) \to (\beta_1,\ldots \beta_n)$ if and only if $\beta_i \ge \alpha_i$ for every $1 \le i \le n$.
\end{itemize} 
\end{rem}

\begin{rem} We also observe a map between two $n$-cubes forms an $(n+1)$-cube and vice-versa.
\end{rem}

\begin{defn} Let $X: \cP(W) \to \cC$ be a $W$-cube in $\cC$ and let $U \subset V \subset W$. The face $\partial^V_U(X)$ is the $(V-U)$-cube defined by
$$T \mapsto X_{T \cup U}, \quad T \subset V-U.$$
That is, $\partial^V_U(X)$ the $(V-U)$-cube formed by all maps in $X$ between $X_U$ and $X_V$.
\end{defn}

\begin{rem} Given a $W$-cube $X:\cP(W) \to \cC$ and $V \subset W$, one can form a $|V|$-subcube via restriction $\cP(V) \to \cP(W) \to \cC$. However, not all subcubes are faces.
\end{rem}

We now restrict our study to cubes valued in model categories of operadic algebras in $\Mod_{Hk}$.

\begin{defn} Let $\cO$ be an operad in $\Mod_{Hk}$, let $X$ be a $W$-cube in the category $\cO \alg(\Mod_{Hk})$, and let $k \in \ZZ$.
\begin{itemize}
\item $X$ is $k$-Cartesian if the map $X_{\emptyset}  \to \holim_{\cP_0(W)}X$ is $k$-connected.
\item $X$ is $\oo$-Cartesian if the map $X_{\emptyset}  \to \holim_{\cP_0(W)}X$ is a weak equivalence.
\item $X$ is strongly $\oo$-Cartesian if each face of dimension $\ge 2$ is $\oo$-Cartesian.
\item $X$ is $k$-coCartesian if the map $\hocolim_{\cP_1(W)}X \to  X_W$ is $k$-connected.
\item $X$ is $\oo$-coCartesian if the map $\hocolim_{\cP_1(W)}X \to X_W$ is a weak equivalence.
\item $X$ is strongly $\oo$-coCartesian if each face of dimension $\ge 2$ is $\oo$-coCartesian.
\end{itemize}
\end{defn}

\begin{rem} A $0$-cube is $k$-Cartesian if its single value is $(k-1)$-connected. A $0$-cube is $k$-coCartesian if its single value is $k$-connected. A $1$-cube is $k$-Cartesian and $k$-coCartesian if it is $k$-connected as a map.
\end{rem}

\begin{defn} The total homotopy fiber of a $W$-cube $X$, denoted $\TotHofib(X)$, is
$$\hofib[X_{\emptyset} \to \holim_{\cP_0(W)}(X)].$$
\end{defn}

The following theorem is called the Higher Blakers-Massey Theorem for structured ring spectra \cite[1.7]{CH2}. 
It is a useful tool for turning coCartesian connectivity estimates for faces of a cube $X$ into a Cartesian connectivity estimate for the entire cube $X$.
\begin{thm} \label{dual Blakers-Massey} Let $\cO$ be an operad in $\Mod_{Hk}$, let $X$ be a $W$-cube in $\cO \alg(\Mod_{Hk})$, assume that $\cO,X_{\emptyset}$ are $(-1)$-connected. Suppose
\begin{itemize}
\item For each nonempty subset $V \subset W$, the $V$-cube $\partial^V_{\emptyset}X$ is $k_V$-coCartesian.
\item We have $-1 \le k_U \le k_V$ for each $U \subset V$.
\end{itemize}
Then $X$ is $k$-Cartesian, where $k$ is the minimum of $-|W| + \sum_{V \in \lambda}k_V$ over all partitions $\lambda$ of $W$ by nonempty subsets not equal to $W$.
\end{thm}

The following theorem is called the Higher Dual Blakers-Massey Theorem for structured ring spectra \cite[1.11]{CH2}.
It is a useful tool for turning Cartesian connectivity estimates for faces of a cube $X$ into a coCartesian connectivity estimate for the entire cube $X$.
\begin{thm} \label{dual Blakers-Massey} Let $\cO$ be an operad in $\Mod_{Hk}$, let $X$ be a $W$-cube in $\cO \alg(\Mod_{Hk})$, assume that $\cO,X_{\emptyset}$ are $(-1)$-connected. Suppose
\begin{itemize}
\item For each nonempty subset $V \subset W$, the $V$-cube $\partial^W_{W-V}X$ is $k_V$-Cartesian.
\item We have $-1 \le k_U \le k_V$ for each $U \subset V$.
\end{itemize}
Then $X$ is $k$-coCartesian, where $k$ is the minimum of $k_W+|W|+1$ and $|W| + \sum_{V \in \lambda}k_V$ over all partitions $\lambda$ of $W$ by nonempty subsets not equal to $W$.
\end{thm}

\begin{defn} If $f:\ZZ \to \ZZ$ is a function, we say that a $W$-cube $X$ is $f$-Cartesian if each $d$-subcube of $X$ is $f(d)$-Cartesian. Similarly,
 we say that a $W$-cube $X$ is $f$-coCartesian if each $d$-subcube of $X$ is $f(d)$-coCartesian.
\end{defn}

We will need a uniformity result for cubical diagrams in $Com \alg(\Mod_{Hk})$. We state it as a proposition for algebras over a general operad $\cO$, where $\cO$ is $(-1)$-connected in each arity. 

\begin{prop} \label{uniformity} Let $k > 0$, then a $W$-cube in $\cO \alg(\Mod_{Hk})$ is $(\id+k)$-Cartesian if and only if it is $(2\id+k-1)$-coCartesian.
\end{prop}
\begin{proof} This follows from \cite[2.4]{D}, \cite[1.7]{CH}, and \cite[1.11]{CH}.
\end{proof}

For the rest of this section we develop further technical language and a result (\ref{messy}) that will be used in our connectivity estimate for commuting a left adjoint past homotopy limits over $\Delta^{\le n}$. 

\begin{defn} \label{infinity Cartesian} Let $Z$ be an objectwise fibrant cosimplicial object in  $\cO \alg(\Mod_{Hk})$, let $n \ge 0$, and denote $[n] := \{0,1,\ldots,n \}$. Let
$$Z:\cP_0([n]) \to \cO \alg(\Mod_{Hk})$$ 
be the composition $\cP_0([n]) \to \Delta^{\le n} \to \Delta \to \cO \alg(\Mod_{Hk})$.  We define $\sC_{n+1}(Z)$, the $\oo$-Cartesian $(n+1)$-cube associated to $Z$, by
$\sC_{n+1}(Z): \cP([n]) \to \cO \alg(\Mod_{Hk})$ where
 \[
    \sC_{n+1}(Z)_V=\left\{
                \begin{array}{ll}
                  \holim_{T \ne \emptyset}Z_T, \quad \text{for} \thickspace V=\emptyset\\
                  Z_V, \quad \text{for} \thickspace V \ne \emptyset
                \end{array}
              \right.
  \]
\end{defn}

\begin{rem} There are weak equivalences
$$\underset{\Delta^{\le n}}{\holim}Z \simeq \holim_{T \ne \emptyset}Z_T \simeq \sC_{n+1}(Z)_{\emptyset}.$$ 
\end{rem}

\begin{rem} A homotopy limit over a diagram indexed by fibrant objects may be computed using the (uncorrected) Bousfield-Kan homotopy limit functor \cite[XI]{BK}.  
\end{rem}

\begin{defn} Let $Z$ be a cosimplicial object in  $\cO \alg(\Mod_{Hk})$, we define the codegeneracy $n$-cube associated to the cosimplicial object $Z$, denoted $\sY_n$, by forming the canonical $n$-cube using the codegeneracy maps
\[\xymatrix{
Z_{[n]} \cdots \ar@<1.0ex>[r] \ar@<0.0ex>[r] \ar@<-1.0ex>[r] & Z_{[2]}   \ar@<-.5ex>[r] \ar@<.5ex>[r] & Z_{[1]} \ar[r] &  Z_{[0]}.
}\]
\end{defn}

The codegeneracy $n$-cube associated to a cosimplicial object $Z$ in $\cO \alg(\Mod_{Hk})$ may be used to compute the layers of the $\Tot$ tower of $Z$.  More precisely, we have:
 
\begin{prop} There are weak equivalences
$$\hofib(\holim_{\Delta^{\le n}}Z \to \holim_{\Delta^{\le n-1}}Z) \simeq \TotHofib(\sY_n)[-n]$$
where $[-n]$ denotes the $n$-fold desuspension in the underlying category $\Mod_{Hk}$.
\end{prop}

\begin{ex} The codegeneracy $2$-cube is pictured as follows:
\[\xymatrix{
Z_{[2]} \ar[r] \ar[d] & Z_{[1]} \ar[d] \\
Z_{[1]} \ar[r] & Z_{[0]}.
}\]
\end{ex}

The following proposition may be found in \cite[7.31]{CH} and will be used in Lemma \ref{use}.
\begin{prop} \label{messy} Let $Z$ be an objectwise fibrant cosimplicial object in  $\cO \alg(\Mod_{Hk})$, let $n \ge 0$, let $\emptyset \ne T \subset [n]$, and $t \in T$. Then there
is a weak equivalence
$$\TotHofib( \partial^T_{\{t\}} \sC_{n+1}(Z)) \simeq \TotHofib(\sY_{|T|-1})[1-|T|]$$
where $\sY_{|T|-1}$ is the codegeneracy $(|T|-1)$-cube associated to $Z$, and $[1-|T|]$ denotes the $(|T|-1)$-fold desuspension in the underlying category $\Mod_{Hk}$.
\end{prop}

\section{The Main Theorem}
In \ref{conservative}, we showed the functor 
$$\Ab(Q-): Ass \alg^{\nccompl}(\Mod_{Hk}) \to K \coalg(Com \alg(\Mod_{Hk}))$$
is conservative. It is natural to ask if $\Ab(Q-)$ furthermore induces an equivalence of homotopical categories.  In this section we prove
that this is the case when restricted to $0$-connected algebras.

\begin{rem} One might like to approach (a version of) the question above using a homotopical version of the Barr-Beck theorem. That is, if one could show that $\Ab(Q-)$ preserves totalizations of $(\Res R)$-split cosimplicial objects, i.e. that there are weak equivalences
$$\Ab (Q \widetilde{\Tot}[Q \Res R K^{\bullet}Y]) \to \widetilde{\Tot}[\Ab (Q Q \Res R K^{\bullet} Y)]$$
for all homotopy $K$-coalgebras $Y$, then the homotopical Barr-Beck theorem would imply there is an equivalence 
$$Ass \alg^{\abcompl}(\Mod_{Hk}) \simeq K \coalg(Com \alg(\Mod_{Hk})).$$
However, the author does not know if the maps
$$\Ab (Q \widetilde{\Tot}[Q \Res R K^{\bullet}Y]) \to \widetilde{\Tot}(\Ab (Q Q \Res R K^{\bullet} Y))$$
are weak equivalences in general. Thus, we base our arguments on connectivity estimates.
\end{rem}

\subsection{Connectivity of the Derived Unit Map}
We first analyze the connectivity of the derived unit map of the
$$Ass \alg (\Mod_{Hk}) \leftrightarrows K \coalg(Com \alg(\Mod_{Hk}))$$
adjunction. 

The following is a version of \cite[2.6]{CH}.
\begin{prop} \label{unit connectivity} If $X \in Ass \alg(\Mod^{> 0}_{Hk})$, then the natural map of $Ass$-algebras 
$$QX \to \underset{\Delta^{\le n}}{\holim}[\Cobar_K^{\bullet}(\Ab(QX))]$$
is $(n+2)$-connected.
\end{prop}
\begin{proof} 
The homotopy fiber of the natural map 
$$QX \to \underset{\Delta^{\le n}}{\holim}[\Cobar_K^{\bullet}(\Ab(QX))]$$
 is equivalent to the total homotopy fiber of the $(n+1)$-cube 
 $$\sX_{n+1}:\Cube_{n+1} \to Ass \alg(\Mod_{Hk})$$ defined by
$$(\alpha_1,\ldots \alpha_{n+1}) \mapsto |\Barc^{\bullet}(\cO_1,Ass, |\Barc^{\bullet}(\cO_2,Ass, \ldots, |\Barc^{\bullet}(\cO_{n+1},Ass,QX)| \ldots)|$$
where 
 \[
    \cO_i=\left\{
                \begin{array}{ll}
                  Ass, \thickspace \text{for} \thickspace \alpha_i=0,\\
                  Com, \thickspace \text{for} \thickspace \alpha_i=1.
                \end{array}
              \right.
  \]

We need to show that if $X$ is 0-connected, then the total homotopy fiber of $\sX_{n+1}$ is $(n+1)$-connected. 
Consider the case $n=0$. Then we are asking that the map 
$$|\Barc^{\bullet}(Ass,Ass,QX)| \to |\Barc^{\bullet}(Com,Ass,QX)|$$ 
be $1$-connected, but this follows from \ref{connectivity estimate}. Now consider the case $n=1$. Then we are asking the total homotopy fiber of the $2$-cube
\[\xymatrix{
|\Barc^{\bullet}(Ass,Ass,|\Barc^{\bullet}(Ass,Ass,QX)|)| \ar[d] \ar[r] & |\Barc^{\bullet}(Com,Ass, |\Barc^{\bullet}(Ass,Ass,QX)|)| \ar[d] \\
|\Barc^{\bullet}(Ass,Ass, |\Barc^{\bullet}(Com,Ass,QX)|)| \ar[r] & |\Barc^{\bullet}(Com,Ass, |\Barc^{\bullet}(Com,Ass,QX)|)|
}\]
be $2$-connected. The total homotopy fiber of the $2$-cube above can be computed by first taking the homotopy fibers of the horizontal maps in the $2$-cube, and then taking the homotopy fiber of the vertical map resulting from that construction. 

We now use an argument similar to \cite[7.13]{CH} to show this is $2$-connected.
Let $r \ge 0$ be an index corresponding to the `outer' geometric realizations in the terms of the $2$-cube. Let $Z := |\Barc^{\bullet}(Com,Ass,QX)|$ and $Z' := |\Barc^{\bullet}(Ass,Ass,QX)|$. Consider the following two rows of fiber sequences
\[\xymatrix{
Ass^{\ge 1} \circ Ass^{\circ r} \circ Z' \ar[r] \ar[d] & Ass \circ (Ass^{\circ r}) \circ Z' \ar[r] \ar[d] & Com \circ Ass^{\circ r} \circ Z' \ar[d] \\
Ass^{\ge 1} \circ Ass^{\circ r}\circ Z \ar[r] & Ass \circ (Ass^{\circ r}) \circ Z \ar[r]  & Com \circ Ass^{\circ r} \circ Z 
}\]

Let $s \ge 0$ be an index corresponding to the `inner' geometric realizations in the terms of the $2$-cube.  Taking the fiber of the left-hand vertical map we have 
$$\text{TotHofib}(\sX_2)_{r,s}=Ass^{\ge 1} \circ Ass^{\circ r} \circ Ass^{\ge 1} \circ Ass^{\circ s} \circ QX.$$
Thus, $\text{TotHofib}(\sX_2)$, after removing the inner and outer realizations, is a bisimplicial object with a $2$-connected term in each bidegree. Taking geometric realizations both vertically and horizontally shows that $\text{TotHofib}(\sX_2)$ is $2$-connected.

The cases for higher values of $n$ are proved similarly.
\end{proof}

\begin{cor} 
If $X \in Ass \alg(\Mod^{> 0}_{Hk})$, then the natural map of $Ass$-algebras 
$$QX \to \widetilde{\Tot}[\Cobar_K^{\bullet}(\Ab(QX))]$$
is a weak equivalence. That is, for connected associative algebras, the derived unit map is a weak equivalence.
\end{cor}

\subsection{Connectivity of the Derived Counit Map}
We now analyze the connectivity of the derived counit map of the
$$Ass \alg (\Mod_{Hk}) \leftrightarrows K \coalg(Com \alg(\Mod_{Hk}))$$
 adjunction. We separate the analysis into two main propositions.

The following is a version of \cite[2.13]{CH}.
\begin{prop} \label{counit connectivity 1} Let $Y \in K \coalg(Com \alg(\Mod^{>0}_{Hk}))$ and $n \ge 0$, then the natural map
$$\Ab (Q \widetilde{\Tot}[\Cobar^{\bullet}_K(Y)]) \to \Ab (Q \underset{\Delta^{\le n}}{\holim}[\Cobar^{\bullet}_K(Y)])$$
is $(n+2)$-connected.
\end{prop}
\begin{proof}
 It suffices to show that the natural map 
$$\underset{\Delta^{\le n+1}}{\holim}[\Cobar_K^{\bullet}(Y)] \to \underset{\Delta^{\le n}}{\holim}[\Cobar_K^{\bullet}(Y)]$$ 
is $(n+2)$-connected.  Computing the homotopy fiber of the natural map above is equivalent to computing the $(n+1)$-fold desuspension of the total homotopy fiber 
of the codegeneracy $(n+1)$-cube 
$$\sY_{n+1}:\Cube_{n+1} \to Com \alg(\Mod_{Hk})$$ 
defined by
$$(\alpha_1,\ldots \alpha_{n+1}) \mapsto |\Barc^{\bullet}(Com,\cO_1,|\Barc^{\bullet}(Com, \cO_2, \ldots, |\Barc^{\bullet}(Com,\cO_{n+1},Q \Res RY)| \ldots)|$$
where 
 \[
    \cO_i=\left\{
                \begin{array}{ll}
                  Ass, \thickspace \text{for} \thickspace \alpha_i=0,\\
                  Com, \thickspace \text{for} \thickspace \alpha_i=1.
                \end{array}
              \right.
  \]
We need to show that if $Y$ is 0-connected, then the total homotopy fiber of $\sY_{n+1}$ is $(2n+2)$-connected. Consider the case $n=0$. Then we are asking that the map
$$|\Barc^{\bullet}(Com,Ass,Q\Res RY)| \to |\Barc^{\bullet}(Com,Com,Q \Res RY)|$$
be $2$-connected. But this follows from \ref{connectivity estimate}. Now consider the case $n=1$. Then we are asking the total homotopy fiber of the $2$-cube
\smallskip
\newline
\resizebox{14cm}{!}{\begin{xy} \newline
\xymatrix{
|\Barc^{\bullet}(Com,Ass,|\Barc^{\bullet}(Com,Ass,Q \Res R Y)|)| \ar[d] \ar[r] & |\Barc^{\bullet}(Com,Com,|\Barc^{\bullet}(Com,Ass,Q \Res R Y)|)| \ar[d] \\
|\Barc^{\bullet}(Com,Ass,|\Barc^{\bullet}(Com,Com,Q \Res R Y)|)| \ar[r] & |\Barc^{\bullet}(Com,Com,|\Barc^{\bullet}(Com,Com,Q \Res R Y)|)|
}
\end{xy}}
\smallskip
\newline
be $4$-connected. The total homotopy fiber of the $2$-cube above can be computed by first taking the homotopy fibers of the horizontal maps in the $2$-cube, and then taking the homotopy fiber of the vertical map resulting from that construction. 

We now use an argument similar to \cite[7.23]{CH} to show this is $4$-connected.
Let $r \ge 0$ be an index corresponding to the `outer' geometric realizations of the terms of the $2$-cube. Let $Z := |\Barc^{\bullet}(Com,Com,Y)|$, and $Z' := |\Barc^{\bullet}(Com,Ass,Y)|$. Consider the following two rows of fiber sequences
\[\xymatrix{
Com \circ (Ass^{\circ r})^{\ge 1} \circ Z' \ar[r] \ar[d] & Com \circ (Ass^{\circ r}) \circ Z' \ar[r] \ar[d] & Com \circ (Com^{\circ r}) \circ Z' \ar[d] \\
Com \circ (Ass^{\circ r})^{\ge 1} \circ Z \ar[r] & Com \circ (Ass^{\circ r}) \circ Z \ar[r]  & Com \circ (Com^{\circ r}) \circ Z 
}\]
where the notation $(Ass^{\circ r})^{\ge 1}$ means that at least one operation in $(Ass^{\circ r})$ is a commutator bracket. 
Let $s \ge 0$ be an index corresponding to the `inner' geometric realizations of the terms of the $2$-cube.  Taking the fiber of the left-hand vertical map we have 
$$\text{TotHofib}(\sY_2)_{r,s}=Com \circ (Ass^{\circ r})^{\ge 1} \circ Com \circ (Ass^{\circ s})^{\ge 1} \circ Y.$$
It follows that for $r,s \ge 0$,
\begin{itemize}
\item $\text{TotHofib}(\sY_2)_{r,s}$ is $2$-connected
\item $\text{TotHofib}(\sY_2)_{0,s} \simeq *$
\item  $\text{TotHofib}(\sY_2)_{r,0} \simeq *$
\end{itemize}
Thus, $\text{TotHofib}(\sY_2)$, after removing the inner and outer realizations, is a bisimplicial object with connectivity properties specified above. Taking geometric realizations both vertically and horizontally shows that $\text{TotHofib}(\sY_2)$ is $4$-connected. 

The cases for higher values of $n$ are proved similarly.
\end{proof}

We now prove a connectivity result for commuting $\Ab(Q-)$ past homotopy limits over $\Delta^{\le n}$.  We rely on the terminology and theory developed in Section \ref{cubical}. We need a few lemmas before the main proposition.

\begin{nota} 
Let $Y$ be a connected homotopy $K$-coalgebra. We let $\sC_{n+1}(Y)$ denote the $\oo$-Cartesian $(n+1)$-cube associated to the objectwise fibrant cosimplicial object $\Cobar^{\bullet}_K(Y)$. See \ref{infinity Cartesian} for this construction.
\end{nota}

\begin{lem} \label{use} Let $n \ge 0$, $\emptyset \ne T \subset [n]$, and $t \in T$. Then the cube
$$\partial^T_{\{t\}} \sC_{n+1}(Y)$$ is $(\id+1)$-Cartesian.
\end{lem}
\begin{proof} Let us first show the special case that the whole cube $\partial^T_{\{t\}} \sC_{n+1}(Y)$ is $|T|$-Cartesian. Note that $|T|=|T-\{t\}|+1$.
It suffices to show that the total homotopy fiber $\TotHofib[\partial^T_{\{t\}} \sC_{n+1}(Y)]$ is $(|T|-1)$-connected.  But this follows from \ref{messy} and that the total homotopy fiber of $\sY_{|T|-1}$ is $(2|T|-1)$-connected, where $\sY_n$ is the codegeneracy $n$-cube associated to the cosimplicial object $\Cobar^{\bullet}_K(Y)$. A similar analysis works for proper subcubes of $\partial^T_{\{t\}} \sC_{n+1}(Y)$, and the result follows.
\end{proof}

\begin{lem} \label{faces} Let $n \ge 0$, $\emptyset \ne T \subset [n]$, and $\emptyset \ne S \subset T$. Then the cube
$$\partial^T_{S} \sC_{n+1}(Y)$$ is $(\id+1)$-Cartesian.
\end{lem}
\begin{proof} We know that the lemma is true for $|S| \le 1$ by the lemma above.  We argue by induction on $|S|$. If $t \in S$, the cube
$\partial^T_{S-\{t\}} \sC_{n+1}(Y)$ may be written as the composition
$$\partial^{T-\{t\}}_{S-\{t\}}\sC_{n+1}(Y) \to \partial^T_{S}\sC_{n+1}(Y).$$
By induction, the composition is $(\id +2)$-Cartesian and the left-hand cube is $(\id+1)$-Cartesian.  Thus, by \ref{uniformity} and \cite[3.8]{CH2}, 
the right-hand cube is $(\id+1)$-Cartesian. 
\end{proof}

\begin{prop} \label{counit connectivity 2} Let $Y \in K \coalg (Com \alg(\Mod^{>0}_{Hk}))$ and $n \ge 1$, then the natural map
$$\Ab (Q \underset{\Delta^{\le n}}{\holim}[\Cobar^{\bullet}_K(Y)]) \to \underset{\Delta^{\le n}}{\holim}[\Ab (Q \Cobar^{\bullet}_K(Y))]$$
is $(n+4)$-connected.
\end{prop}
\begin{proof} 
Let $Y \in K \coalg (Com \alg(\Mod^{>0}_{Hk}))$ and let 
$\sC_{n+1}(Y)$ be the $\oo$-Cartesian $(n+1)$-cube built from the cosimplicial object $\Cobar^{\bullet}_K(Y)$.
We wish to show that $\Ab [Q \sC_{n+1}(Y)]$ is an $(n+4)$-Cartesian diagram in $Com \alg(\Mod^{>0}_{Hk})$. 
We will prove the stronger result that $\Ab [Q \sC_{n+1}(Y)]$ is $(\id+3)$-Cartesian.  However, by \ref{uniformity}, this is equivalent to showing that
$\Ab [Q \sC_{n+1}(Y)]$ is $(2\id+2)$-coCartesian. Hence, it suffices to show that $\sC_{n+1}(Y)$ is $(2\id+2)$-coCartesian.
For $\emptyset \ne V \subset [n]$, the $V$-cube $\partial^{[n]}_{[n]-V} \sC_{n+1}(Y)$ is $(\id+1)$-Cartesian by \ref{faces}. This implies that
$\sC_{n+1}(Y)$ is $(2\id+2)$-coCartesian by \ref{dual Blakers-Massey}. This is the desired result.
\end{proof}

\begin{cor} 
If $Y \in K \coalg (Com \alg(\Mod^{>0}_{Hk}))$, the natural map
$$\Ab (Q \widetilde{\Tot}[\Cobar^{\bullet}_K(Y)]) \to \widetilde{\Tot}[\Ab (Q \Cobar^{\bullet}_K(Y))] \to Y$$
is a weak equivalence. That is, for connected homotopy $K$-coalgebras, the derived counit map is a weak equivalence.
\end{cor}
\begin{proof} It suffices to show the second map is a weak equivalence, but this follows from the tower $\{\Ab Q \Cobar^{\bullet}_K(Y) \}$ having an extra degeneracy.
\end{proof}

\subsection{Proof of the Main Theorem and Corollaries}
We now prove the main theorem of this chapter.
\begin{thm} \label{main} There is an equivalence of homotopical categories
$$Ass \alg^{\nccompl}(\Mod^{> 0}_{Hk}) \simeq K \coalg(Com \alg(\Mod^{> 0}_{Hk})).$$ 
\end{thm}
\begin{proof} It suffices to show there is an equivalence 
$$Ass \alg(\Mod^{> 0}_{Hk}) \simeq K \coalg(Com \alg(\Mod^{> 0}_{Hk}))$$
and then use \ref{connected is complete} to obtain the result. We will show that
\begin{enumerate}
\item If $X \in Ass \alg(\Mod^{> 0}_{Hk})$, then the natural map of $Ass$-algebras 
$$QX \to \widetilde{\Tot}[\Cobar^{\bullet}_K\Ab(QX)]$$
is a weak equivalence.
\item If $Y \in K \coalg(Com \alg(\Mod^{> 0}_{Hk}))$, then the natural map of homotopy $K$-coalgebras
$$\Ab Q \widetilde{\Tot}[\Cobar^{\bullet}_KY] \to Y$$
is a weak equivalence.
\end{enumerate}
The first of these weak equivalences follows from \ref{unit connectivity} while the second follows from \ref{counit connectivity 1} and \ref{counit connectivity 2}.
\end{proof}

\begin{cor} An algebra $X \in Ass \alg(\Mod^{> 0}_{Hk})$ can be recovered from its derived abelianization $Com \overset{\LL}{\circ}_{Ass}(X)$ together with its homotopy $K$-coalgebra structure.  
\end{cor}

\begin{cor} There is an equivalence
$$Ass \alg^{\nccompl}(\Mod^{>0}_{Hk}) \simeq Ass \alg^{\abcompl}(\Mod^{>0}_{Hk}).$$
\end{cor}


\section{Derived NC-complete $\cE_n$-algebras and Relative Koszul Duality}
We develop analogues of the main results above in the context of $\cE_n$-algebras, where $n \ge 2$. Many of the arguments can be applied in this new context without change, and we provide proofs when the connectivity arguments do slightly change.

\subsection{The Abelianization-Restriction Quillen Adjunction for $\cE_n$-Algebras}
The canonical map of operads $\cE_n \to Com$ induces a restriction functor $$\Res: Com \alg(\Mod_{Hk}) \to \cE_n \alg(\Mod_{Hk}).$$ This restriction functor has a left adjoint, abelianization, which is computed via the relative composition product $\Ab(-) := Com \circ_{\cE_n}(-)$. Furthermore, there is a Quillen adjunction
\[\xymatrix{
Com \alg(\Mod_{Hk}) \ar[r]_{\Res} & \cE_n \alg(\Mod_{Hk}) \ar@/_1pc/[l]_{Com \circ_{\cE_n}(-)}.
}\]

\begin{defn} The derived abelianization of $X \in \cE_n \alg(\Mod_{Hk})$ is the derived relative composition product $Com \overset{\LL}{\circ}_{\cE_n}(X) \simeq \Ab(QX)$.
\end{defn}

\begin{nota} We write $K_n: Com \alg(\Mod_{Hk}) \to Com \alg(\Mod_{Hk})$ for the point-set homotopical comonad $\Ab Q \Res R$, and we write
$T_n: \cE_n \alg(\Mod_{Hk}) \to \cE_n \alg(\Mod_{Hk})$ for the point-set homotopical monad $\Res R \Ab Q$.
\end{nota}

\begin{rem} Thus, there is well-defined notion of a homotopy $K_n$-coalgebra in $Com \alg(\Mod_{Hk})$ and there is a simplicial $A_{\oo}$-category of homotopy $K_n$-coalgebras in $Com \alg(\Mod_{Hk})$.
\end{rem}

\begin{defn} Let $Y$ be a homotopy $K_n$-coalgebra.  The cobar construction on $Y$, $\Cobar^{\bullet}_{K_n}(Y)$, is the cosimplicial object in $\cE_n \alg(\Mod_{Hk})$ given by
\[\xymatrix{
Q \Res RY \ar@<-.5ex>[r] \ar@<.5ex>[r] & Q \Res R K_n Y \ar@<1.0ex>[r] \ar@<0.0ex>[r] \ar@<-1.0ex>[r] & Q \Res R K_n^2 Y  \cdots
}\]
\end{defn}

\begin{defn} The $\Ab$-completion of an $\cE_n$-algebra $X$ is the $\Tot$ of a Reedy fibrant replacement of the cobar construction on the homotopy $K_n$-coalgebra $\Ab(QX)$.
That is, $\widetilde{\Tot}$ of the cosimplicial object
\[\xymatrix{
Q\Res R \Ab (QX) \ar@<-.5ex>[r] \ar@<.5ex>[r] & QK_n \Res R \Ab (QX) \ar@<1.0ex>[r] \ar@<0.0ex>[r] \ar@<-1.0ex>[r] & QK_n^2\ \Res R \Ab (QX)  \cdots
}\]
or of the equivalent cosimplicial object,
\[\xymatrix{
QT_n(X) \ar@<-.5ex>[r] \ar@<.5ex>[r] & QT_n^2(X) \ar@<1.0ex>[r] \ar@<0.0ex>[r] \ar@<-1.0ex>[r] & QT_n^3(X)  \cdots
}\]
We denote the $\Ab$-completion of $X$ by $\hat{X}_{\Ab}$. 
\end{defn}

We now define the derived unit map and derive counit map for the 
$$\cE_n \alg (\Mod_{Hk}) \leftrightarrows K_n \coalg(Com \alg(\Mod_{Hk}))$$
adjunction in the context of $\cE_n$-algebras.  

\begin{defn} The derived unit map is the natural map
$$\eta: QX \to \widetilde{\Tot}(Q \Res R K_n^{\bullet} \Ab(QX))$$
which is induced from the canonical map $QX \to Q \Res R \Ab(QX)$.
\end{defn}

\begin{defn} We say that an $\cE_n$-algebra $X$ is $\Ab$-complete if the derived unit map $$\eta: QX \to \widetilde{\Tot}(Q \Res R K_n^{\bullet} \Ab(QX))$$ is a weak equivalence. We denote the full subcategory of 
$\Ab$-complete $\cE_n$-algebras by $\cE_n^{\abcompl} \alg(\Mod_{Hk})$.
\end{defn}

\begin{prop} The functor $\Ab(Q-)$ determines a fully-faithful embedding of the homotopy category of $\Ab$-complete $\cE_n$-algebras into the homotopy category of homotopy $K_n$-coalgebras.
\end{prop}

We now define the derived counit map for the 
$$\cE_n \alg (\Mod_{Hk}) \leftrightarrows K_n \coalg(Com \alg(\Mod_{Hk}))$$
adjunction in the context of $\cE_n$-algebras.  

\begin{defn} Let $Y \in K_n \coalg(Com \alg(\Mod_{Hk}))$, the derived counit map is the natural map of homotopy $K_n$-coalgebras 
$$\epsilon: \Ab Q \widetilde{\Tot}[\Cobar_{K_n}^{\bullet}(Y)] \to Y.$$
This consists of simplicial sets
$$\epsilon_l:\Delta^l \to \Hom_{Com}(\Ab Q \widetilde{\Tot}[Q \Res R K_n^{\bullet}Y], R K_n^lY)$$
that are the composites
\begin{align*}
\Delta^l &\to \Hom_{\cE_n}(\widetilde{\Tot}[Q \Res R K_n^{\bullet}Y],Q \Res R K_n^l Y) \\
&\to \Hom_{Com}(\Ab Q \widetilde{\Tot}[Q \Res R K_n^{\bullet}Y],\Ab Q Q \Res R K_n^l Y) \\
&\to \Hom_{Com}(\Ab Q \widetilde{\Tot}[Q \Res R K_n^{\bullet}Y], R K_n^l Y)
\end{align*}
where the first map is adjoint to the projection
$$\widetilde{\Tot}[Q \Res R K_n^{\bullet}Y] \to \Hom(\Delta^l,Q \Res R K_n^l Y)$$
the second map uses the simplicial enrichment of $\Ab Q$, and the third map involves the counit maps
$$\Ab QQ \Res R \to \Ab \Res R \to R.$$
\end{defn}

\subsection{Commutator Complete $\cE_n$-Algebras}
We now develop an analogue of the commutator (co)filtration in the context of $\cE_n$-algebras.

\begin{defn} There is a cofiltration on the operad $\cE_n$ in $\Mod_{Hk}$:
$$\cE_n \rightarrow \ldots \rightarrow  \cE_n^{\le k} \rightarrow \ldots \rightarrow \cE_n^{\le 2} \rightarrow \cE_n^{\le 1} \rightarrow \cE_n^{\le 0} \simeq Com,$$
where $\cE_n^{\le k}$ is defined by the $k^{th}$-Postnikov truncation (or $k^{th}$-t-structure truncation) of each arity of the $\cE_n$ operad. Here, we are referring to the canonical t-structure on $\Mod_{Hk}$.
\end{defn}

\begin{nota} We also let
\begin{itemize}
\item $\cE_n^{=k} := \hofib(\cE_n^{\le k} \to  \cE_n^{\le k-1})$
\item $\cE_n^{>k} := \hofib(\cE_n \to  \cE_n^{\le k}).$
\end{itemize}
\end{nota}

\begin{rem} We will use the fact that $\cE_n^{> k}(m)$ is $k$-connected for each arity $m$. and that $\cE_n^{=k}(m)$ is $(k-1)$-connected for each arity $m$.  
\end{rem}

\begin{rem} The descending filtration corresponding to this cofiltration is precisely the filtration on the $\cE_n$ operad whose associated graded operad is the shifted Poisson operad $\cP_n$. That is, an algebra over the operad $\cP_n$ is a commutative algebra equipped with a Poisson bracket of cohomological degree $1-n$.
\end{rem}

\begin{defn} \label{derived En-cofiltration}
The derived commutator cofiltration on $X \in \cE_n \alg(\Ch_k)$ is the tower
$$X \to \cdots \to {\cE_n}^{\le k} \overset{\LL}{\circ}_{\cE_n}(X) \rightarrow \cdots \rightarrow \cE_n^{\le 1} \overset{\LL}{\circ}_{\cE_n}(X) \rightarrow Com \overset{\LL}{\circ}_{\cE_n}(X).$$
\end{defn}

\begin{defn} An algebra $X \in \cE_n \alg(\Mod_{Hk})$ is called commutator nilpotent of degree k if 
$$X \simeq \holim [\cE_n^{\le k} \overset{\LL}{\circ}_{\cE_n}(X) \rightarrow \ldots \rightarrow \cE_n^{\le 1} \overset{\LL}{\circ}_{\cE_n}(X) \rightarrow Com \overset{\LL}{\circ}_{\cE_n}(X)]$$
\end{defn}

\begin{defn} The commutator completion of an algebra $X \in \cE_n \alg(\Mod_{Hk})$ is 
$$\hat{X}_{\cmtr} := \underset{k}{\holim} [\cE_n^{\le k} \overset{\LL}{\circ}_{\cE_n}(X)]$$
\end{defn}

\begin{defn} \label{commutator complete} An algebra $X \in \cE_n \alg(\Mod_{Hk})$ is called commutator complete if there is a weak-equivalence
$$X \simeq \hat{X}_{\cmtr}.$$
We let $\cE_n \alg^{\cmtrcompl}(\Mod_{Hk})$ denote the full subcategory of $\cE_n \alg(\Mod_{Hk})$ formed by commutator complete algebras.
\end{defn}

The following proposition is our first connection between the commutator filtration on $\cE_n$-algebras and relative Koszul duality.
\begin{prop} \label{conservativeEn} The functor $$\Ab(Q-): \cE_n \alg^{\cmtrcompl}(\Mod_{Hk}) \to K_n \coalg(Com \alg(\Mod_{Hk}))$$ is conservative. 
\end{prop}
\begin{proof} Let $f: X \to X'$ be a map of commutator complete $\cE_n$-algebras such that 
$$Com \overset{\LL}{\circ}_{\cE_n}(X) \to Com \overset{\LL}{\circ}_{\cE_n}(X')$$
 is a weak equivalence. We must show that this implies that $f: X \to X'$ is also a weak equivalence.  Since $X$ and $X'$ are commutator complete, we have 
$$X \simeq \underset{j}{\holim} [\cE_n^{\le j} \overset{\LL}{\circ}_{\cE_n}(X)] \quad X' \simeq \underset{j}{\holim} [\cE_n^{\le j} \overset{\LL}{\circ}_{\cE_n}(X')].$$
We proceed by induction on $j$ using the fiber sequences 
$$\cE_n^{= j} \overset{\LL}{\circ}_{\cE_n}(X) \to \cE_n^{\le j} \overset{\LL}{\circ}_{\cE_n}(X) \to \cE_n^{\le j-1} \overset{\LL}{\circ}_{\cE_n}(X).$$
The right $\cE_n$-action map $\cE_n^{= j} \circ \cE_n \to \cE_n^{=j}$ factors through the commutative operad: 
\[\xymatrix{
\cE_n^{= j} \circ \cE_n \ar[r] \ar[d] & \cE_n^{=j} \\
 \cE_n^{=j} \circ Com \ar[ur] 
}\]
Thus, we have weak equivalences
$$\cE_n^{= j} \overset{\LL}{\circ}_{\cE_n}(X) \simeq \cE_n^{= j} \overset{\LL}{\circ}_{Com} (Com \overset{\LL}{\circ}_{\cE_n}(X))$$
and hence fiber sequences
$$\cE_n^{= j} \overset{\LL}{\circ}_{Com} (Com \overset{\LL}{\circ}_{\cE_n}(X)) \to \cE_n^{\le j} \overset{\LL}{\circ}_{\cE_n}(X) \to \cE_n^{\le j-1} \overset{\LL}{\circ}_{\cE_n}(X).$$
So if 
$$Com \overset{\LL}{\circ}_{\cE_n}(X) \to Com \overset{\LL}{\circ}_{\cE_n}(X')$$ 
is a weak equivalence, then
$$\cE_n^{\le j} \overset{\LL}{\circ}_{\cE_n}(X) \to \cE_n^{\le j} \overset{\LL}{\circ}_{\cE_n}(X')$$ is a weak equivalence for every $j \ge 0$.
The conclusion follows.
\end{proof}

\subsection{Connectivity Estimates for $\cE_n$-Algebras}

We now develop connectivity estimates in the context of $\cE_n$-algebras to prepare for the main theorem. Since the connectivity of $\cE_n$-algebras behaves slightly differently than $Ass$-algebras we provide proofs of these statements.
\begin{lem} \label{connectivity estimate En} If $X \in \cE_n \alg(\Mod^{> 0}_{Hk})$, then
\begin{itemize}
\item $\cE_n^{\le k} \overset{\LL}{\circ}_{\cE_n}(X)$ is 0-connected.
\item $\cE_n^{= k} \overset{\LL}{\circ}_{\cE_n}(X)$ is $k$-connected.
\item $\cE_n^{> k} \overset{\LL}{\circ}_{\cE_n}(X)$ is $(k+1)$-connected.
 \end{itemize}
\end{lem}

We now show that connected $\cE_n$-algebras are commutator complete.
\begin{prop} \label{connected is complete En}  There is an equivalence 
$$\cE_n \alg^{\cmtrcompl}(\Mod^{> 0}_{Hk}) \simeq \cE_n \alg(\Mod^{> 0}_{Hk}).$$ 
\end{prop}
\begin{proof} We wish to show that if $X \in  \cE_n \alg(\Mod^{> 0}_{Hk})$, then 
$$QX \simeq  \underset{k}{\holim} [\cE_n^{\le k} \overset{\LL}{\circ}_{\cE_n}(X)].$$
It suffices to show that the natural map
$$\underset{k}{\holim} [\cE_n \overset{\LL}{\circ}_{\cE_n}(X)] \to  \underset{k}{\holim} [\cE_n^{\le k} \overset{\LL}{\circ}_{\cE_n}(X)]$$
is a weak equivalence, where the left hand side is a homotopy limit over a constant diagram. Consider the following commutative diagram

\[\xymatrix{
\pi_i \underset{k}{\holim} [\cE_n \overset{\LL}{\circ}_{\cE_n}(X)] \ar[r] \ar[d] & \pi_i \underset{k}{\holim} [\cE_n^{\le k} \overset{\LL}{\circ}_{\cE_n}(X)] \ar[d] \\
\underset{k}{\holim} [\pi_i (\cE_n \overset{\LL}{\circ}_{\cE_n}(X))] \ar[r] & \underset{k}{\holim} [\pi_i (\cE_n^{\le k} \overset{\LL}{\circ}_{\cE_n}(X))]
}\]

We wish to show that if $X$ is connected, the top horizontal arrow is a weak equivalence. We will do this by showing the other three arrows in the commutative diagram are weak equivalences. It is easy to see that the left hand vertical map is a weak equivalence: indeed ${\lim}^1\pi_{i+1}|\Barc(\cE_n,\cE_n,X)|=0$ since we have a constant tower.
The bottom horizontal arrow is a weak equivalence by \ref{connectivity estimate En}. Similarly, it follows from \ref{connectivity estimate En} that 
$\pi_i |\Barc(\cE_n^{\le k+1},\cE_n,X)| \to \pi_i |\Barc(\cE_n^{\le k},\cE_n,X)|$ is an isomorphism for $i \le k+1$ and surjective for $i=k+2$. Hence,
$\pi_i |\Barc(\cE_n^{\le k+1},\cE_n,X)|$ is eventually constant, so ${\lim}^1\pi_{i+1}|\Barc(\cE_n^{\le k},\cE_n,X)|=0$, and thus the right hand vertical arrow is a weak equivalence.  The result follows.
\end{proof}

We analyze how the derived abelianization functor interacts with (highly) connected $\cE_n$-algebras. 
\begin{thm} If $X$ and $Y$ are cofibrant $0$-connected $\cE_n$-algebras, then
\begin{enumerate}[label*=\arabic*.]
\item An $\cE_n$-algebra $X$ is $l$-connected if and only if $Com \overset{\LL}{\circ}_{\cE_n}(X)$ is $l$-connected.
\item If $Com \overset{\LL}{\circ}_{\cE_n}(X)$ is $l$-connected, then the map $\pi_k (X) \to \pi_k (Com \overset{\LL}{\circ}_{\cE_n}(X))$ is an isomorphism for $k \le l+1$ and surjective for $k=l+2$.
\item A map $f:X \to Y$ is $l$-connected if and only if the induced map \\ $Com \overset{\LL}{\circ}_{\cE_n}(X) \to Com \overset{\LL}{\circ}_{\cE_n}(Y)$ is $l$-connected
\end{enumerate}
\end{thm}

\begin{proof} Assume that $Com \overset{\LL}{\circ}_{\cE_n}(X)$ is $l$-connected. Then $|\Barc^{\bullet}(Com,\cE_n,QX)|$ is $l$-connected and 
$|\Barc^{\bullet}(\cE_n^{=k},\cE_n,QX)|$ is $(k+l)$-connected for $k \ge 1$. Hence, the maps 
$$\pi_i(\cE_n^{\le k} \overset{\LL}{\circ}_{\cE_n}(X)) \to \pi_i(\cE_n^{\le k-1} \overset{\LL}{\circ}_{\cE_n}(X))$$
 are isomorphisms for $i \le k+l$ and surjections for $i=k+l+1$. In particular, for each $i \le l+1$ the tower 
$\{ \pi_i(\cE_n^{\le k} \overset{\LL}{\circ}_{\cE_n}(X)) \}$ is a tower of isomorphisms, and since $Com \overset{\LL}{\circ}_{\cE_n}(X)$ is $l$-connected, it follows that each stage of the tower is $l$-connected. Since $X$ is $0$-connected, the map $\pi_i(X) \to \pi_i(\cE_n^{\le k} \overset{\LL}{\circ}_{\cE_n}(X))$ is an isomorphism for $i \le k+1$. Taking $k$ sufficiently large, implies that $X$ is $l$-connected.

Conversely, assume that $X$ is $l$-connected. Then $|\Barc^{\bullet}(\cE_n^{\le k},\cE_n,QX)|$ is $l$-connected and both $|\Barc^{\bullet}(\cE_n^{>k-1},\cE_n,QX)|$ and
 $|\Barc^{\bullet}(\cE_n^{=k},\cE_n,QX)|$ are $(k+l)$-connected for $k \ge 1$. Thus,  the natural maps 
 $$\pi_i(X) \to \pi_i(\cE_n^{\le k-1} \overset{\LL}{\circ}_{\cE_n}(X))$$
 and 
 $$\pi_i(\cE_n^{\le k} \overset{\LL}{\circ}_{\cE_n}(X)) \to \pi_i(\cE_n^{\le k-1} \overset{\LL}{\circ}_{\cE_n}(X))$$
 are isomorphisms for $i \le k+l$ and surjections for $i=k+l+1$. Thus, 
 $$\pi_i(X) \to \pi_i(Com \overset{\LL}{\circ}_{\cE_n}(X))$$
 is an isomorphism for $i \le l+1$ and a surjection for $i=l+2$. Since $X$ is $l$-connected, it follows that $Com \overset{\LL}{\circ}_{\cE_n}(X)$ is $l$-connected.
 
 Now, let $f: X \to Y$ be a map $\cE_n$-algebras such that 
 $$Com \overset{\LL}{\circ}_{\cE_n}(X) \to Com \overset{\LL}{\circ}_{\cE_n}(Y)$$
 is an $l$-connected map. Consider the commutative diagram
\[\xymatrix{
|\Barc^{\bullet}(\cE_n^{=k},\cE_n,X)| \ar[r] \ar[d] & |\Barc^{\bullet}(\cE_n^{\le k},\cE_n,X)| \ar[r] \ar[d] & |\Barc^{\bullet}(\cE_n^{\le k-1},\cE_n,X)| \ar[d] \\
|\Barc^{\bullet}(\cE_n^{=k},\cE_n,Y)| \ar[r] & |\Barc^{\bullet}(\cE_n^{\le k},\cE_n,Y)| \ar[r] & |\Barc^{\bullet}(\cE_n^{\le k-1},\cE_n,Y)|. 
}\]
Since the left hand and right hand vertical maps are $l$-connected for $k \ge 2$, the middle vertical map is $l$-connected for $k \ge 2$.  Thus 
$f:X \to Y$ is an $l$-connected map.

Conversely, if $f:X \to Y$ is an $l$-connected map, then 
$$Com \overset{\LL}{\circ}_{\cE_n}(X) \to Com \overset{\LL}{\circ}_{\cE_n}(Y)$$
is also an $l$-connected map.
\end{proof}

\begin{cor}
 A map $f:X \to Y$ of $0$-connected $\cE_n$-algebras is a weak equivalence if and only if  the induced map $Com \overset{\LL}{\circ}_{\cE_n}(X) \to Com \overset{\LL}{\circ}_{\cE_n}(Y)$ is a weak equivalence.
\end{cor}

\subsection{The Main Theorem for $\cE_n$-Algebras and Corollaries}

We now record the main theorem for $\cE_n$-Algebras.
\begin{thm} \label{mainEn} There is an equivalence of homotopical categories
$$\cE_n \alg^{\cmtrcompl}(\Mod^{> 0}_{Hk}) \simeq K_n \coalg(Com \alg(\Mod^{> 0}_{Hk})).$$ 
\end{thm}

\begin{cor} An algebra $X \in \cE_n \alg(\Mod^{> 0}_{Hk})$ can be recovered from its derived abelianization $Com \overset{\LL}{\circ}_{\cE_n}(X)$ together with its homotopy $K_n$-coalgebra structure.  
\end{cor}

\begin{cor} There is an equivalence
$$\cE_n \alg^{\cmtrcompl}(\Mod^{>0}_{Hk}) \simeq \cE_n \alg^{\abcompl}(\Mod^{>0}_{Hk}).$$
\end{cor}


\end{document}